\documentclass[]{interact}

\usepackage{xr-hyper}
\usepackage{hyperref}
\usepackage{lipsum}
\usepackage[normalem]{ulem}
\newcommand{\stkout}[1]{\ifmmode\text{\sout{\ensuremath{#1}}}\else\sout{#1}\fi}
\usepackage{amsfonts}
\usepackage{graphicx}
\usepackage{epstopdf}
\usepackage{amsmath,mathtools}
\usepackage{algorithm}
\usepackage[noend]{algpseudocode}
\usepackage{amsmath,amssymb,amsfonts,bm}
\usepackage[inline]{enumitem}   
\usepackage{tcolorbox}
\usepackage{mathtools}
\newcommand{\st}{\text{s.t.}}

\newcommand{\unitVec}{\mathbf{e}}

\newcommand{\xttt}{\vartheta}
\newcommand{\RR}{\mathbb{R}}
\newcommand{\NN}{\mathbb{N}}
\newcommand{\vecOnes}{\mathbf{1}}
\newcommand{\vecZeros}{\mathbf{0}}

\newcommand{\DDD}{\mathcal{D}}

\newcommand{\SSS}{\mathcal{S}}

\newcommand{\AAA}{\mathcal{A}}

\newcommand{\HHH}{\mathcal{H}}
\newcommand{\LLL}{\mathcal{L}}

\newcommand{\CCC}{\mathcal{C}}

\newcommand{\EEE}{\mathcal{E}}

\newcommand{\ttt}{\theta}

\newcommand{\tr}{\text{\rm tr}}
\newcommand{\trc}{{\text{\rm tr}^*}}
\newcommand{\apx}{\text{\rm app}}
\newcommand{\mc}{\text{\rm mc}}
\newcommand{\ef}{\text{\rm ef}}
\newcommand{\eg}{\text{\rm eg}}

\newcommand{\thresh}{\text{\rm thr}}
\newcommand{\fcd}{\text{\rm fcd}}

\newcommand{\inc}{\text{\rm inc}}
\newcommand{\dec}{\text{\rm dec}}
\newcommand{\bhm}{\text{\rm bhm}}
\newcommand{\enl}{\text{\rm enl}}

\newcommand{\alg}{\text{\rm Alg}}

\newcommand{\nulll}{\text{\rm null}}

\newcommand{\proj}[2]{\textrm{proj}_{#1}\left(#2\right)} 
\newcommand{\Footnote}[1]{}
\DeclareMathOperator*{\argmax}{arg\,max}

\DeclareMathOperator{\app}{app}

\makeatletter
 
 \newcommand{\algparboxtwo}[1]{\parbox[t]{\dimexpr\linewidth-\algorithmicindent-\algorithmicindent}{#1\strut}}
  \newcommand{\algparboxthree}[1]{\parbox[t]{\dimexpr\linewidth-\algorithmicindent-\algorithmicindent-\algorithmicindent}{#1\strut}}
    
\ifpdf
  \DeclareGraphicsExtensions{.eps,.pdf,.png,.jpg}
\else
  \DeclareGraphicsExtensions{.eps}
\fi

\theoremstyle{plain}% Theorem-like structures provided by amsthm.sty
\newtheorem{theorem}{Theorem}[section]
\newtheorem{lemma}[theorem]{Lemma}

\newtheorem{assumption}[theorem]{Assumption}

\theoremstyle{definition}
\newtheorem{definition}[theorem]{Definition}

\theoremstyle{remark}
\newtheorem{remark}{Remark}

\author{
\name{Frank~E.~Curtis \textsuperscript{a}, Shima Dezfulian \textsuperscript{b*}, \thanks{*Corresponding author. Email: dezfulian@u.northwestern.edu.} 
\and Andreas W\"achter. \textsuperscript{b}}
\affil{\textsuperscript{a} Department of ISE, Lehigh University, Bethlehem, PA;\\  \textsuperscript{b} Department of IEMS, Northwestern University, Evanston, IL}
}
\usepackage{amsopn}

\allowdisplaybreaks

\begin{document}

\title{% Exploiting Prior Function Evaluations in Derivative-Free Optimization
{Derivative-Free Bound-Constrained Optimization for Solving Structured Problems with Surrogate Models}}

\maketitle

% REQUIRED
\begin{abstract}
    We propose and analyze a model-based derivative-free (DFO) algorithm for solving bound-constrained optimization problems where the objective function is the composition of a smooth function and a vector of black-box functions.  We assume that the black-box functions are smooth and the evaluation of them is the computational bottleneck of the algorithm. The distinguishing feature of our algorithm is the use of approximate function values at interpolation points which can be obtained by an application-specific surrogate model that is cheap to evaluate.  As an example, we consider the situation in which a sequence of related optimization problems is solved and present a regression-based approximation scheme that uses function values that were evaluated when solving prior problem instances. In addition, we propose and analyze a new algorithm for obtaining interpolation points that handles unrelaxable bound constraints. Our numerical results show that our algorithm outperforms a state-of-the-art DFO algorithm for solving a least-squares problem from a chemical engineering application when a history of black-box function evaluations is available.
\end{abstract}

% REQUIRED

% REQUIRED

\begin{amscode}
  49M15, 65K05, 65K10, 90C30, 90C56
\end{amscode}

%******
% Body
%******
%******
% Body
%******
%*********
% Section
%*********
\section{Introduction}\label{sec.intro}

We propose an algorithmic framework for solving bound-constrained optimization problems that have two distinct characteristics. First, we consider problems that involve objective functions that are smooth (i.e., continuously differentiable), yet derivatives are intractable to compute.  We refer to such functions as \emph{black-box} functions, as is common in the literature.  Typically, for a black-box function, evaluating the objective function itself is very expensive, and to speed up the optimization our framework permits the use of surrogate models that provide a fast but approximate evaluation of the objective function at a subset of the evaluation points.  Second, we consider problems for which the objective function is a composition of a smooth function (whose analytical form is known) and (expensive-to-evaluate) black-box functions, and our framework is designed to take advantage of the knowledge of this structure.

An example class of problems for which our framework is applicable are multi-output simulation optimization problems, where one aims to minimize a cost calculated from a vector of outputs of a simulation; see, e.g., \cite{bates1988nonlinear, choi2020gradient, dolan2004benchmarking}.  The black-box functions that arise in these contexts could, for example, be computed by numerical simulations, say that solve discretized partial differential equations (PDEs). Here, a natural surrogate model might be based on solving the PDEs to lower accuracy \cite{liu2016multi, zhang2021multi}.

Another setting in which our framework is applicable, which we explore in detail in this paper, is the situation in which a sequence of optimization problems is solved where the instances are distinguished by exogenous hyperparameters of the black-box functions that take on different but similar values.  In this context, we utilize regression-based surrogate models constructed from function outputs that have been collected during the solution of previous problem instances.  Sequences of related optimization problems occur in various applications throughout science, engineering, and economics.  Here, we mention  a setting that has motivated this work, namely, nonlinear least-square problems that commonly appear in data-fitting applications \cite{bates1988nonlinear, dennis1996numerical}. For such problems, one aims to find the hyperparameters for the black-box functions that map inputs to outputs as accurately as possible based on a set of observed data.

Our proposed algorithm builds upon state-of-the-art model-based DFO methods for solving problems that possess a composite structure of the objective.  In particular, inspired by \cite{cartis2019derivative, wild2017chapter, zhang2010derivative}, our algorithm computes trial steps by minimizing local models of the objective, where the composite structure of the objective is exploited by combining local models that are constructed for each black-box function separately.  It has been shown, e.g., in the context of nonlinear least squares, that algorithms that exploit such structure often perform better than those that do not; see, e.g., \cite{cartis2019derivative, wild2017chapter, zhang2010derivative}.

%************
% Subsection
%************
\subsection{Contributions}

We propose and analyze a model-based DFO algorithm for solving bound-constrained optimization problems that exploits approximate function values. This allows the algorithm to avoid new evaluations of black-box functions at interpolation points when the function values at these points can be approximated sufficiently accurately by using  an application-specific surrogate model or by exploiting previously evaluated black-box functions at nearby points. Typically, previously evaluated black-box functions are made available through prior runs of the  DFO algorithm for solving related problems or prior iterations of the run of the algorithm. We propose a regression-based function approximation scheme that is built by using nearby points when previously evaluated black-box functions are available.
Moreover, for handling bound constraints, we introduce a new subroutine  that iteratively adds feasible points to the interpolation set and ensures that a required geometry condition is satisfied for finding model parameters in a numerically stable manner.   Our numerical experiments show that our algorithm outperforms a state-of-the-art model-based DFO method when solving a sequence of related least-squares problems, especially when the budget for function evaluations is small, which is typically the case in real-world applications.  {More specifically, we compare our proposed algorithm, referred to as $\alg_{\HHH}$, and an algorithm that has all of the same features of our algorithm except that does not utilize prior function evaluations to {approximate} function values at interpolation points, referred to as $\alg_{\emptyset}$. Hence, similar to existing model-based DFO algorithms, $\alg_{\emptyset}$ utilizes function evaluations that are obtained in prior iterations by possibly choosing them as candidate interpolation points in subsequent iterations, but it does not use such information for approximation.}

%************
% Subsection
%************
\subsection{Notation}\label{sec:notation}

The sets of real numbers, $n$-dimensional real vectors, and $n$-by-$m$-dimensional real matrices are denoted by $\RR$, $\RR^n$, and $\RR^{n \times m}$, respectively.  The sets of nonnegative and positive real numbers are denoted by $\RR_{\geq0}$ and $\RR_{>0}$, respectively.  The vector of all zeros is denoted as $\vecZeros$, the identity matrix is denoted as $I$, the vector of all ones is denoted as $\vecOnes$, and the $i$th unit vector is denoted as $\unitVec_i$, where in each case the size of the object is determined by the context.  The set of nonnegative integers is denoted as $\NN := \{0,1,\dots\}$ and we define $[n] := \{1,\dots,n\}$ for any $n \in \NN \setminus \{0\}$.

%For a set $\SSS \subseteq \RR^n$ and real number $c \in \RR$, we define $c \bullet \SSS := \{cs : s \in \SSS\}$.  
For sets $\SSS_1 \subseteq \RR^n$ and $\SSS_2 \subseteq \RR^n$, we define Minkowski addition and subtraction in the usual manner, e.g., $\SSS_1 + \SSS_2 = \{s_1 + s_2 : s_1 \in \SSS_1, s_2 \in \SSS_2\}$.  That said, in the particular case when one of the sets is a singleton, say $\{c\}$ with $c \in \RR^n$, then we simply write $c + \SSS = \{ c + s : s \in \SSS\}$.  The cardinality of a set $\SSS$ is denoted by $|\SSS|$.

Given any real number $q \geq 1$, the $\ell_q$-norm of a vector $v \in \RR^n$ is written as~$\|v\|_q$.  The closed $\ell_q$-norm ball with center $x \in \RR^n$ and radius $\Delta \in \RR_{\geq0}$ is denoted as $B_{q}(x, \Delta) := \{\bar x: \|\bar x - x \|_q \leq \Delta \}$.  The dual norm of $\|\cdot\|_q$ is denoted and defined by $\|z\|_{q^*} := \max \{z^Tx : x \in B_{q}(0,1)\}$.  The Frobenius-norm of a matrix $M$ is denoted by $\|M\|_{\rm F}$. For a matrix $M \in \RR^{m \times n}$ and real numbers $p \geq 1$ and $q \geq 1$, the  $L_{p, q}$-norm of $M$ is written as $\| M\|_{p, q} = \big(\sum_{j=1}^n (\sum_{i=1}^m |M_{ij}|^p)^{\frac{q}{p}} \big)^{\frac{1}{q}}$ and the $(p,q)$-operator norm of $M$ is written as  $\|M\|_{(p, q)} = \sup_{\|x\|_p \leq 1} \| M x\|_q$, although, for shorthand, $\| \cdot\|_p = \| \cdot \|_{(p, p)}$.

{For the sake of generality, a pair of norms used throughout the paper are those defining the \emph{trust region radius} and the \emph{approximation radius} (i.e., precision parameter) in our algorithm, which we respectively denote as $\|\cdot\|_{\tr}$ and $\|\cdot\|_{\app}$ for some real numbers $\tr \geq 1$ and $\app \geq 1$}.  In our analysis, we make use of norm equivalence constants for finite-dimensional real vector spaces; in particular, due to the equivalence of all such norms, for any positive integers $n_x$ and $p$, there exist constants $(\kappa_{\tr_0},\kappa_{\tr_1},\kappa_{\tr_0^*},\kappa_{\tr_1^*},\kappa_{\tr_2^*},\kappa_{\tr_3^*},\kappa_{\app_0^*}) \in \RR_{\geq 0}^7$ such that
\begin{subequations}\label{eq.equivconstants}
	\begin{align}
	\|v\|_2 &\leq \kappa_{\tr_0} \|v\|_{\tr} && \text{for all $v \in \RR^{n_x}$}, \label{eq.equiv7} \\
	\|v\|_{\tr} &\leq \kappa_{\tr_1} \|v\|_2 && \text{for all $v \in \RR^{n_x}$}, \label{eq.equiv8} \\
	\|v\|_{\trc} &\leq \kappa_{\tr_0^*} \|V^{-1}\|_2 \|Vv\|_{\infty} && \text{for all $(v,V) \in \RR^{n_x} \times \RR^{n_x \times n_x}$}, \label{eq.equiv1} \\
	\|Vv\|_{\trc} &\leq \kappa_{\tr_1^*} \|V\|_{\trc,1} \|v\|_2 && \text{for all $(v,V) \in \RR^p \times \RR^{n_x \times p}$}, \label{eq.equiv2} \\ 
	\|VUV^T\|_{(\tr,\trc)} &\leq \kappa_{\tr_2^*} \|V\|_{\trc,1}^2 \|U\|_2 && \text{for all $(V,U) \in \RR^{n_x \times p} \times \RR^{p \times p}$}, \label{eq.equiv3}\\
	\| v \|_{\tr} \| v \|_{\trc} &\leq \kappa_{\tr_3^*} \| v\|_2^2&& \text{for all $v \in \RR^{n_x}$}, \label{eq.equiv4} \\
	\text{and}\ \ 
	\|Vv\|_{\app^*} &\leq \kappa_{\app_0^*} \| v \|_{\app} \|V\|_2 &&\text{for all $(v,V) \in \RR^{n_x} \times \RR^{n_x \times n_x}$}. \label{eq.equiv5}  
	\end{align}
\end{subequations}

For differentiable $g: \RR^n \to \RR$ and $y \in \RR^n$, we use $\partial_i g(y)$ and $\partial_i \partial_j g(y)$ to denote $\frac{\partial g(y)}{\partial y_i}$ and $\frac{\partial^2 g(y)}{\partial y_i \partial y_j}$ for any $(i,j) \in [n] \times [n]$, respectively.  The gradient function of $g$ is $\nabla g: \RR^n \to \RR^n$, where $[\nabla g(y)]_i = \partial_i g(y)$ for any $i \in [n]$ and $y \in \RR^n$.  If $g$ is twice differentiable, then the Hessian function of $g$ is $\nabla^2 g: \RR^n \to \RR^{n \times n}$, where $[\nabla^2 g(y)]_{i, j} = \partial_i \partial_j g(y)$ for any $(i,j) \in [n] \times [n]$ and $y \in \RR^n$.  For differentiable vector-valued $G: \RR^n \to \RR^m$, the transpose of the Jacobian function of $G$ is denoted by $\nabla G: \RR^n \to \RR^{n \times m}$, where $[\nabla G(y)]_{i,j} = \partial_i G_j(y)$ for any $(i,j) \in [n] \times [m]$ and $y \in \RR^n$.

Given any $M \in \RR^{m \times n}$, its column space is $\text{span}(M)$, its null space is $\nulll(M)$, and given any $v \in \RR^m$, the projection of $v$ onto $\text{span}(M)$ is $\proj{M}{v}$.

One of the main quantities in the paper is a function $F$ that takes two input arguments.  For the purposes of designing our optimization algorithm, it is convenient to write the domain of the function as $\RR^{n_x + n_\ttt}$, although for ease of exposition we write the function and its gradient as taking inputs in $\RR^{n_x} \times \RR^{n_\ttt}$; in particular, these are written in the form $F(x;\ttt)$ and $\nabla F(x; \ttt)$. This should not lead to confusion due to the natural one-to-one mapping between elements of $\RR^{n_x} \times \RR^{n_\ttt}$ and $\RR^{n_x + n_\ttt}$.

%************
% Subsection
%************
\subsection{Problem Formulation}\label{sec:prob_formulation}

Formally, we consider the minimization problem
\begin{equation}\label{eq:gen_problem_no_seq}
\min_{x \in \Omega} f(x) \ \ \text{with} \ \ f(x) = h(F(x)) 
\end{equation}
where the function $h: \RR^{p} \to \RR$ is \emph{glass-box} in the sense that its analytical form is known and $\Omega = [x_L, x_U]$ for some $x_L \in (\RR \cup \{-\infty\})^{n_x}$ and $x_U \in (\RR \cup \{\infty\})^{n_x}$ with $x_{L,j} < x_{U,j}$ for all $j \in [n_x]$.
The vector-valued function $F: \RR^{n_x} \to \RR^p$ consists of black-box \emph{element functions}  $F_i$ for each $i \in [p]$, that is
\begin{equation}\label{eq:vec_F_no_seq}
F(x) = \begin{bmatrix} F_1(x)& \cdots & F_p(x) \end{bmatrix}^T.
\end{equation}
Our framework does not evaluate $F$ outside of $\Omega$, i.e., the bound constraints defining $\Omega$ are
unrelaxable constraints.
We conjecture that our proposed techniques can be extended to the setting when $\Omega$ is any closed convex set by following the ideas in \cite{hough2021model}. 

\subsection{Multi-Output Simulation-based Optimization with Surrogate Models}
\label{sec.multioutput_surrogate_formulation}

As previously mentioned, an example setting in which our framework is applicable is multi-output simulation-based optimization with surrogate models.  We assume that exact evaluations of $F$ are expensive and hence inexact evaluations of $F$ that are obtained by computationally cheaper surrogate models may be beneficial. For all $i\in[p]$, we use $\tilde F_i(x, \delta)$ to denote an approximation of $F_i(x)$ that, for a given precision parameter  $\delta \in \RR_{>0 }$, satisfies the accuracy requirement
\begin{equation}\label{eq:gen_err_ub_0_no_seq}
|F_i(x) - \tilde F_i(x, \delta)| \leq \kappa_{\app} \delta
\end{equation}
for some fixed $\kappa_{\apx} \in  \RR_{\geq 0}$.
In the algorithm, $\delta$ will be proportional to $\Delta^2$, where $\Delta$ is a trust region radius. Typically,   $\delta$ becomes smaller the closer the iterates are to  {a stationary point}.  
Exact evaluations of $F$ at iterates and trial points are still required.

We envision that the approximation $\tilde F_i$ is obtained from a surrogate model that is a cheap approximation of the expensive function $F_i$.  Typically, the surrogate model will be tailored to the specific application.  For instance, if $F_i$ is computed from the numerical solution of PDEs, then the surrogate model might simply be defined by solving the PDEs to a lower accuracy that depends on the precision parameter $\delta$ as long as it is possible to bound the error proportional to $\delta$ so that \eqref{eq:gen_err_ub_0_no_seq} holds\Footnote{AW: Do we have a reference for that?}.  Importantly, the actual value of $\kappa_{\app}$ need not be known by our algorithmic framework.

\subsection{Sequence of  Optimization Problems}
\label{sec.sequence_formulation}

Another setting in which our algorithm is applicable is when a sequence of optimization problems, indexed by $t \in \NN$, of the form
\begin{equation}\label{eq:gen_problem}
\min_{x \in \Omega_t} f_t(x)\ \ \text{with}\ \ f_t(x) \coloneqq h_t(F(x, \theta_t))
\end{equation}
is to be solved. The function $F: \RR^{n_x} \times \Theta \to \RR^p$, is defined similar to \eqref{eq:vec_F_no_seq}, i.e.,
\begin{equation}
\label{eq:vec_F}
F(x, \ttt_t) = \begin{bmatrix}
F_1(x, \ttt_t) & \dots & F_p(x, \ttt_t)
\end{bmatrix}^T,
\end{equation}
and  $\ttt_t \in \Theta$ for some $\Theta \subseteq \RR^{n_\ttt}$  are some exogenous hyperparameters of $F$, which could be, for example, settings inside a simulator that remain fixed during the solution of the $t$-th problem but might vary across instances.
The quantities $h_t$ and  $\Omega_t$ are defined similar to $h$ and $\Omega$ in problem \eqref{eq:gen_problem_no_seq} but may vary with $t$.

Since our presumption is that evaluations of $F$ are expensive, it may be beneficial to store and make use of function values that were computed during previous runs of the optimization algorithm.  That is, for all $t>0$, we presume that a history of prior black-box function values is available \emph{a priori}.  We denote the set of prior information for element function~$i$ when solving problem $t$ (obtained when solving problems~0 through $t-1$ and during prior iterations when solving problem~$t$) by~$\HHH_{i,t}$, which is a finite set and might be defined differently depending on the setting.

For instance, in a setting where \eqref{eq:gen_problem} computes a cost function from a simulation that has $p$ different outputs,
we might store all prior evaluations of the outputs, i.e.,
\begin{align}
\label{eq:history_multi_output}
{\HHH_{i, t} = \left\{(x, \ttt_j, F_i(x, \ttt_j)): F_i(x, \ttt_j) \text{ has been evaluated, with} \ j \in \{0, \dots, t-1\}\right\}}.
\end{align}
If we then need an approximation $\tilde F_i(\bar x,\theta_t, \delta)$ of $F_i(\bar x,\theta_t)$ for a specific $x=\bar x$ that satisfies a condition similar to \eqref{eq:gen_err_ub_0_no_seq}, namely,
\begin{equation}\label{eq:gen_err_ub_0_seq}
|F_i(x, \ttt_t) - \tilde F_i(x, \ttt_t, \delta)| \leq \kappa_{\app} \delta,
\end{equation}
we could build a regression model based on all points in $\{(x,\theta):(x,\theta,F_i(x, \ttt))\in\HHH_{i,t}\}$ within some approximation radius of $(\bar x, \theta_t)$ that depends on $\delta$; we will make this more concrete in Section~\ref{sec.multioutput}.
If an insufficient number of such points are available, then $F_i(\bar x,\theta_t)$ can be computed exactly as its own approximation.

Optimizing a sequence of optimization problems commonly appears in data-fitting applications. Suppose we have a simulator $\phi(x, w) \to \RR$ with internal parameters $x \in \Omega$ and we want to choose $x$ so that the output $\phi(x, w_i)$ for a feature vector $w_i$ best matches an observation $y_i$. More precisely, given data $\left\{ (w_{i, t}, y_{i, t}), \dots, (w_{p, t}, y_{p, t})\right\}$, we are interested in solving a  least-squares problem of the form
\begin{align}\label{eq:opt_least_square}
\min_{x \in \Omega_t}\ \frac{1}{2} \sum_{i=1}^p \left(\phi(x, w_{i, t}) - y_{i, t} \right)^2,\ \ \text{where $\Omega_t$ represents element-wise bounds}.
\end{align}
One can formulate this problem as an instance of problem \eqref{eq:gen_problem} by defining
\begin{align*}
\theta_t \coloneqq &\ \begin{bmatrix} w_{1, t}^T & w_{2, t}^T & \dots & w_{p, t}^T \end{bmatrix}^T, & F_i(x, \ttt_t) \coloneqq &\ \phi(x, w_{i, t})\ \text{for all $i \in [p]$}, \\
y_t \coloneqq &\ \begin{bmatrix} y_{1, t} & \dots & y_{p, t} \end{bmatrix}^T, & \text{and}\ \ h_t(v) \coloneqq &\ \tfrac{1}{2}\|v - y_t\|^2\ \text{for all $v \in \RR^p$}.
\end{align*}
Notice that since $F_i(\cdot,\theta_t) \equiv \phi(\cdot,w_{i,t})$ for all $(i,t) \in [p] \times \NN$, this means that one may approximate the value of \emph{any} component of $F$ by exploiting \emph{all} prior evaluations of~$\phi$, even those obtained for different components.  In other words, the history is the same for all components of $F$ in the sense that $\HHH_{i,t} = \HHH_{j,t}$ for all $(i,j) \in [p] \times [p]$, where 
\begin{align}\label{eq:history_least_square}
{\HHH_{i, t} = \left\{(x,w_{j,k},\phi(x, w_{j, k})) : \phi(x, w_{j, k}) \text{ has been evaluated, with} \ j \in [p], k \in \{0, \dots, t-1 \}\right\}}.
\end{align}

%************
% Subsection
%************
\subsection{Literature Review}
\label{sec.literature_review}
Generally speaking, DFO algorithms can be categorized as direct-search, finite-difference-based, or model-based methods.  Direct-search methods, such as pattern~search methods \cite{audet2004pattern, gray2004appspack} and the Nelder-Mead method \cite{nelder1965simplex}, are often outperformed by finite-difference-based and model-based methods \cite{more2009benchmarking} when one is minimizing a smooth objective.  That said, one of the strengths of direct-search methods is that they are more readily applicable when an objective function is nonsmooth \cite{audet2006mesh,custodio2008using} or even discontinuous \cite{vicente2012analysis}.  {A pattern-search approach for solving structured and partially separable, nonsmooth, and potentially discontinuous functions is proposed in \cite{porcelli2022exploiting}. It is assumed that the objective function can be expressed as a sum of \emph{element functions} with each element function depending on a relatively small subvector of the decision variables. The proposed algorithm exploits the structure of the problems to improve the scalability of pattern-search methods for solving such problems.} Finite-difference approaches approximate derivatives using finite difference schemes, which are then embedded within a gradient-based optimization approach, such as a steepest descent or quasi-Newton method; see, e.g., \cite{shi2021numerical}.  Empirical evidence has shown that finite-difference methods can be competitive with model-based methods, at least when one presumes no noise in the function values.  
Our proposed method falls into the model-based method category; hence, we provide a more comprehensive overview of them in the remainder of this section.

Model-based trust-region methods for unconstrained optimization have received a lot of attention in the literature; see, e.g., \cite{conn1997recent,conn2008geometry0, conn2009introduction,marazzi2002wedge,powell2002uobyqa,powell2004least, powell2006newuoa, wild2008mnh}.  Such methods operate by constructing in each iteration a local multivariate model of the objective function.  Typically, linear or quadratic interpolation models are used.  The resulting model is minimized within a trust region to compute a trial point, then common trust-region strategies for updating the iterate and trust region radius are employed. Examples of theoretical results on the convergence of model-based methods to first- or second-order stationary points can be found in \cite{conn1997convergence, conn2009introduction}.  In order to guarantee convergence, the interpolation points used for building the models need to satisfy a geometry condition referred to as \emph{well-poisedness}. See   \cite{conn1997recent,conn2008geometry0, conn2009introduction,marazzi2002wedge, wild2008mnh} for further discussion. Linear or quadratic models can also be obtained by regression instead of interpolation; see \cite{conn2008geometry}.  Other types of models have been used as well.  For example, interpolating radial basis function models have been used in ORBIT \cite{wild2008orbit, wild2011global}, a model-based trust-region algorithm with theoretical convergence guarantees.  One can also distinguish model-based methods in terms of the type of accuracy requirement that are imposed on the function values; many require function values that satisfy an accuracy requirement surely  \cite{bellavia2018levenberg,conn1997recent,conn2008geometry0,conn2009introduction,wild2008mnh} while others require accurate function values at least with some positive probability \cite{bandeira2014convergence, chen2018stochastic,gratton2018complexity}.  Our algorithm falls into the former category since we impose function value accuracy requirements that must hold surely.

Model-based algorithmic ideas have been extended to solve constrained optimization problems as well.  In \cite{conn1998derivative}, a high-level discussion on how to handle various types of constraints is provided.  BOBYQA \cite{powell2009bobyqa} is designed to solve bound-constrained problems that are unrelaxable.  COBYLA \cite{powell1994direct} can solve inequality constrained problems by constructing linear interpolation models for the objective and constraint functions using points that lie on a simplex.  In \cite{conejo2013global} and \cite{hough2021model}, DFO methods for solving problems with closed convex constraints are studied, where it is assumed that projections onto the feasible region are tractable.  The algorithms in \cite{conejo2013global} and \cite{hough2021model} have global convergence guarantees.  CONORBIT \cite{regis2017conorbit} is an extension of ORBIT \cite{wild2008orbit} that handles relaxable inequality constraints and unrelaxable bound constraints.

One of the motivating settings for our proposed algorithm is least-squares optimization.  Hence, we mention that for the special case of such problems, tailored algorithms that take into account the structure of the problem have been proposed previously.  For example, DFO-LS \cite{cartis2019improving} and DFO-GN \cite{cartis2019derivative} suggest constructing separate linear models for the residuals, then exploiting a Gauss-Newton approach in order to construct a quadratic model for the least-squares objective function.  DFO-LS \cite{cartis2019improving} can handle bound constraints as well.  {In cases where the function and gradient can be evaluated with dynamic accuracy, least-squares problems have been studied by \cite{bellavia2018levenberg}. The authors propose a Levenberg–Marquardt method to solve such problems, where it is assumed that the values of the function and gradient can be made as accurate as needed.}
Constructing separate quadratic interpolation models for the residuals  and using Taylor approximation to construct a quadratic model for the least-squares objective has been considered in POUNDERS~\cite{wild2017chapter} and DFLS~\cite{zhang2010derivative}.

Most relevant for this paper is the fact that all of the aforementioned methods do not consider how to exploit  approximate function values at interpolation points that can be obtained by a computationally cheap surrogate model \cite{forrester2007multi, liu2016multi, pousa22multitask, zhang2021multi} or by exploiting nearby prior function values. Although, in the context of bilevel DFO optimization, the idea of re-using prior function values has been investigated in \cite{conn2012bilevel}, no specific mechanism for approximating function values is provided.

%************
% Subsection
%************
\subsection{Organization}
Our generic DFO algorithmic framework is presented in Section~\ref{sec.gen_framework}. In Section~\ref{sec.intrplset}, an algorithm for obtaining well-poised interpolation sets that takes into account bounds is provided. In Section~\ref{sec:fcn_approx}, we provide an approximation scheme that uses prior function evaluations such that one obtains an implementable instance of the framework.  Numerical experiments are presented in Section~\ref{sec.numerical}.
% and concluding remarks in Section~\ref{sec.conclusion}.

%*********
% Section
%*********

\section{A Derivative-Free Algorithmic Framework}
\label{sec.gen_framework}

A generic model-based algorithmic framework for solving problem \eqref{eq:gen_problem_no_seq} is presented in this section.  
It can employ approximate function values in the run of the algorithm as long as certain accuracy requirements are imposed.

In Section~\ref{sec.fl_models}, we describe a procedure for the construction of \emph{fully linear} element function models and fully linear models of the overall function.   In Section \ref{sec.gen_alg_desc}, a description of the  algorithm is provided with pseudocode presented in Algorithm~\ref{alg.dfo_approx}; see page~\pageref{alg.dfo_approx}. We sketch a convergence proof for Algorithm~\ref{alg.dfo_approx} in Section~\ref{sec.gen_conv_analysis} and Appendix~\ref{app:conv_proof} of the supplementary material.

Algorithm~\ref{alg.dfo_approx} is a trust-region algorithm that maintains and updates a trust region radius $\Delta_k$ that is bounded by $\Delta_{\max} \in \RR_{>0}$.
Denoting the $f(x_0)$-sublevel set for the objective function~$f$ as $\LLL_0 = \{x \in \RR^{n_x}: f(x) \leq f(x_0)\}$, it follows by construction that the iterate sequence $\{x_k\}$ is contained in the enlarged set
\begin{equation*}
\LLL_{\enl} := \LLL_0 + B_{\tr}(0, \Delta_{\max}).
%   + B_{\app}(0, \delta_{\max}). 
\end{equation*}
Throughout this section, we assume the following about the functions defining~$f$.
{Recall that  $\|\cdot\|_{\tr}$ is used to define the trust region radius in our algorithm,  where $\tr$ is a real number greater than one.}

\begin{assumption}\label{assumption:F_h}
	There exists an open convex set ${\cal X}$ containing $\LLL_{\enl} \cap \Omega$ over which, for each $i \in [p]$, one has that $F_i(\cdot)$ is continuously differentiable, $F_i(\cdot)$ is Lipschitz continuous with constant $L_{F_i,x} \in \RR_{>0}$ such that
	\begin{equation*}
	|F_i(x) - F_i(\bar x)| \leq L_{F_i, x} \| x - \bar x \|_{\tr} \ \  \text{for all} \ \ (x, \bar x) \in {\cal X} \times {\cal X},
	\end{equation*}
	and $\nabla F_i(\cdot)$ is Lipschitz continuous with constant $L_{\nabla F_i,x} \in \RR_{>0}$ such that
	\begin{equation*}
	\|\nabla F_i(x) - \nabla F_i(\bar{x})\|_{\trc} \leq L_{\nabla F_i,x} \|x - \bar{x}\|_{\tr}\ \ \text{for all}\ \ (x,\bar{x}) \in {\cal X} \times {\cal X}.
	\end{equation*}
	In addition, there exists an open convex set ${\cal Y}$ containing $\{F(x): x \in  \LLL_{\enl} \cap \Omega\}$ over which one has that $h$ is twice-continuously differentiable, $h$ is Lipschitz continuous with constant $L_h \in \RR_{>0}$ such that
	\begin{equation}\label{eq.hfuncLip}
	|h(y) - h(\bar{y})| \leq L_h \|y - \bar{y}\|_2\ \ \text{for all}\ \ (y,\bar{y}) \in {\cal Y} \times {\cal Y},
	\end{equation}
	$\nabla h$ is Lipschitz continuous with constant $L_{\nabla h} \in \RR_{>0}$ such that
	\begin{equation}\label{eq.hLip}
	\|\nabla h(y) - \nabla h(\bar{y})\|_2 \leq L_{\nabla h} \|y - \bar{y}\|_2\ \ \text{for all}\ \ (y,\bar{y}) \in {\cal Y} \times {\cal Y},
	\end{equation}
	there exists $\kappa_{\nabla h} \in \RR_{>0}$ such that $\|\nabla h(y)\|_2 \leq \kappa_{\nabla h}$ for all $y \in {\cal Y}$, and for each $i \in [p]$ there exists $\kappa_{\partial_i h} \in \RR_{>0}$ such that $|\partial_i h(y)| \leq \kappa_{\partial_i h}$ for all $y \in {\cal Y}$.
\end{assumption} 
In addition, defining
\begin{align}\label{eq:F_x_lipschitz_consts}
L_{F_x} = \sum_{i=1}^p L_{F_{i, x}},\ \ 
\bar L_{F_x} = \left(\sum_{i=1}^p L_{F_{i, x}}^2 \right)^{\frac{1}{2}},\ \ \text{and}\ \ 
L_{\nabla F_x} = \sum_{i=1}^p L_{\nabla F_{i, x}},
\end{align}
let us also mention that, under Assumption~\ref{assumption:F_h}, it follows from \eqref{eq:F_x_lipschitz_consts}  that $f$ is continuously differentiable and with \eqref{eq.equivconstants} one can show that
\begin{align}
\|\nabla f(x) - \nabla f(\bar{x})\|_{\trc}
\leq&\ \kappa_{\tr_1^*} \left( \kappa_{\tr_3^*} L_{F_{x}} L_{\nabla h}  \bar L_{F_{i, x}} + \kappa_{\nabla h}  L_{\nabla F_{x}}  \right) \| x - \bar x \|_{\tr} \label{eq.needthis}
\end{align}
for all $(x,\bar{x}) \in {\cal X} \times {\cal X}$, where ${\cal X}$ is defined as in the assumption, which is to say that the gradient function $\nabla f$ is Lipschitz continuous over ${\cal X}$ with constant as shown.

%************
% Subsection
%************
\subsection{Fully Linear Models}\label{sec.fl_models}

Model-based derivative-free optimization methods often use local linear or quadratic interpolation models \cite{conn2009introduction, wild2008orbit}.  More generally, in iteration $k \in \NN$, let $\DDD_k$ be a set of interpolation directions in a neighborhood of the current iterate~$x_k$, e.g., each member of the set $\DDD_k$ has the form $x - x_k$ for some $x \in B_{\tr}(x_k, \Delta_k) \cap \Omega$.  Considering an element function $F_i$ for some $i \in [p]$, one commonly lets $q_{k,i}$ denote a local interpolation model of $F_i$ around $x_k$ satisfying $q_{k, i}(x_k + d) = F_i(x_k + d)$ for all $d \in \DDD_k$.  However, our algorithm merely requires
\begin{equation}
\label{eq:gen_approx_intrpl_cond}
q_{k, i}(x_k + d) = \tilde F_i(x_k + d, \delta_k)\ \ \text{for all}\ \ d \in \DDD_k,
\end{equation}
where the approximate function value $\tilde F_i(x_k + d, \delta_k)$ is required to satisfy a condition similar to  \eqref{eq:gen_err_ub_0_no_seq} for all $d \in \DDD_k$; see Lemma~\ref{lemma:comp_fl} in this subsection.

We make the following assumption about $q_{k,i}$ for all $k \in \NN$ and $i \in [p]$. Let us define $q_k(x) := \begin{bmatrix} q_{k, 1}(x) & \dots & q_{k, p}(x) \end{bmatrix}^T$.  For the functions $h$ and $\nabla h$, this assumption should be seen to augment Assumption~\ref{assumption:F_h}, e.g., with respect to the definitions of the Lipschitz constants.

\begin{assumption}\label{assumption:q}
	There exists an open convex set $\cal X$ containing $\LLL_{\enl} \cap \Omega$ over which, for all $k \in \NN$ and $i \in [p]$, one has that $q_{k,i}$ is twice-continuously differentiable, $\nabla q_{k, i}$ is Lipschitz continuous with constant $L_{\nabla q_i} \in \RR_{>0}$ such that
	\begin{equation*}
	\|\nabla q_{k, i}(x) - \nabla q_{k, i} (\bar{x})\|_{\trc} \leq L_{\nabla q_i} \|x - \bar{x}\|_{\tr}\ \ \text{for all}\ \ (x, \bar x) \in {\cal X} \times {\cal X},
	\end{equation*}
	and there exists $\kappa_{\nabla q_{i}} \in \RR_{>0}$ such that $\|\nabla q_{k, i}(x)\|_2 \leq \kappa_{\nabla q_i}$ for all $x \in {\cal X}$.  In addition, there exists an open convex set ${\cal Y}$ containing $\bigcup_{k=1}^\infty \{q_k(x): x \in \LLL_{\enl} \cap \Omega\}$ over which $h$ is twice-continuously differentiable, $h$ is Lipschitz continuous with constant $L_h \in \RR_{>0}$ such that \eqref{eq.hfuncLip} holds, $\nabla h$ is Lipschitz continuous with constant $L_{\nabla h} \in \RR_{>0}$ such that \eqref{eq.hLip} holds, there exists $\kappa_{\nabla h} \in \RR_{>0}$ such that $\|\nabla h(y)\|_2 \leq \kappa_{\nabla h}$ for all $y \in {\cal Y}$, and for each $i \in [p]$ there exists $\kappa_{\partial_i h} \in \RR_{>0}$ such that $|\partial_i h(y)| \leq \kappa_{\partial_i h}$ for all $y \in {\cal Y}$.
\end{assumption}
A practical method for constructing a model $q_{k, i}$ that satisfies Assumption~\ref{assumption:q} is provided in \cite{stefan2008derivative, wild2008mnh}.  For our purposes later on, let us define
\begin{align}\label{eq:q_consts}
L_{\nabla q} = \sum_{i=1}^p L_{\nabla q_i}\ \ \text{and}\ \ \kappa_{\nabla q} = \sum_{i=1}^p \kappa_{\nabla q_i}.
\end{align}

In model-based derivative-free optimization, since the derivative of a black-box function cannot be evaluated analytically and Taylor models are replaced with interpolation or regression models, one needs to ensure that the model is accurate enough in a neighborhood of the current iterate, commonly defined by the trust region.  This is achieved by providing bounds on the error in the model and its gradient.  The concept of a \emph{fully linear} model defines such a situation common in modern DFO methods.

\begin{definition}
	A sequence of differentiable models $\{m_k\}$, with $m_k : \RR^n \to \RR$ for all $k \in \NN$, is fully linear with respect to a differentiable function $f : \RR^n \to \RR$ over $B_{\tr}(x_k, \Delta_k) \cap \Omega$ with constants $(\kappa_{\ef},\kappa_{\eg}) \in \RR_{\geq0} \times \RR_{\geq0}$ if, for all $x \in B_{\tr}(x_k, \Delta_k) \cap \Omega$,
	% , and all $k \in \NN$
	\begin{align}\label{eq:main_thm_fcn_bnd}
	|m_k(x) - f(x) | \leq \kappa_{\ef} \Delta_k^2\ \ \text{and}\ \ \| \nabla m_{k}(x) - \nabla f(x) \|_{\trc} \leq \kappa_{\eg} \Delta_k \ \ \text{for all}\ \ k \in \NN.
	\end{align} 
\end{definition}

The following lemma provides sufficient conditions for the models $\{q_{k,i}\}_{k\in\NN}$ to be fully linear with respect to $F_i(\cdot)$ over a trust region for all $i \in [p]$.

\begin{lemma}\label{lemma:comp_fl}
	Suppose that Assumptions~\ref{assumption:F_h} and \ref{assumption:q} hold and consider arbitrary $k \in \NN$ and ${\cal D}_k := \{d_1, \dots, d_{n_x}\} \subset B_{\tr}(0, \Delta_k) \cap (\Omega - x_k)$ such that $(a)$ for all $i \in [p]$, the model $q_{k,i}$ satisfies~\eqref{eq:gen_approx_intrpl_cond} on ${\cal D}_k \cup \{0\}$, and $(b)$ the matrix $\begin{bmatrix}  d_1 & \dots & d_{n_x} \end{bmatrix}$ is invertible and, for some $\Lambda \in \RR_{>0}$ and $\bar \kappa_{\apx} \in \RR_{>0}$, one has that
	\begin{align}
	\left\| \begin{bmatrix}  d_1 & \dots & d_{n_x} \end{bmatrix}^{-1} \right \|_2 &\leq \frac{\Lambda}{ \Delta_k}\ \ \text{and} \label{cond:fl-1} \\
	|F_i(x_k + d) - \tilde F_i(x_k+d, \delta_k) | &\leq \bar \kappa_{\apx} \Delta_k^2 \ \ \text{for all}\ \ d \in {\cal D}_k \cup \{0\}. \label{cond:fl-2}
	\end{align}
	Then, for any $x \in B_{\tr}(x_k; \Delta_k) \cap \Omega$ and any $i \in [p]$, the model $q_{k, i}$ satisfies
	\begin{align}
	|q_{k, i}(x) - F_i(x) | &\leq \hat \kappa_{\ef} \Delta_k^2 \label{eq:fl_m} \\ \text{and}\ \ 
	\| \nabla q_{k,i}(x) - \nabla F_i(x) \|_{\trc} &\leq \hat \kappa_{\eg} \Delta_k \label{eq:fl_g},
	\end{align}
	for some $\hat \kappa_{\ef} \in \RR_{>0}$ and $\hat \kappa_{\eg} \in \RR_{>0}$ independent of $k$.
\end{lemma}
% \begin{proof}
One can prove Lemma~\ref{lemma:comp_fl} is a similar manner as in the proof of Theorem~4.1 in~\cite{stefan2008derivative}. The main difference is the use of approximate function values, instead of exact function values,  at interpolation points. See Appendix~\ref{app:lemmas_proof} of the supplementary material for
a complete proof.
% \end{proof}  

With local models for all of the element functions, a local quadratic model for $f$ around $x_k$ can be obtained using a Taylor approximation.  For example, let us consider the model $m_k$ for each $k \in \NN$ as the second-order Taylor-series approximation
%with $q_k(x) := \begin{bmatrix} q_{k, 1}(x) & \dots & q_{k, p}(x) \end{bmatrix}^T$
%can be expressed as
\begin{multline}\label{eq:master_model}
m_k(x) = h(q_k(x_k)) + \nabla h(q_k(x_k))^T \nabla q_k(x_k)^T (x-x_k) \\
+ \tfrac{1}{2} (x - x_k)^T \left(\sum_{i=1}^p \partial_i h(q_k(x_k))\nabla^2 q_{k, i}(x_k)+ \nabla q_k(x_k) \nabla^2 h(q_k(x_k)) \nabla q_k(x_k)^T\right) (x - x_k).
\end{multline}
Under our stated assumptions, the second-order derivatives of this model are bounded and the models $\{m_k\}$ are fully linear with respect to $f$ within a trust region.  These facts are stated formally in the next two lemmas; see Appendix~\ref{app:lemmas_proof} of the supplementary material for proofs.

\begin{lemma}\label{lemma:bounded_hessian_m}
	Suppose that Assumption~\ref{assumption:q} holds.  Then, there exists $\kappa_{\bhm} \in \RR_{> 0}$ such that $\|\nabla^2 m_k(x_k) \|_{(\tr, \trc)} \leq \kappa_{\bhm}$ for all $k \in \NN$.
\end{lemma}

\begin{lemma}
	\label{lemma:master_fl}
	Suppose that Assumptions~\ref{assumption:F_h} and \ref{assumption:q} hold and consider arbitrary $k \in \NN$ and ${\cal D}_k := \{d_1, \dots, d_{n_x}\} \subset B_{\tr}(0, \Delta_k) \cap (\Omega - x_k)$ such that $(a)$ for all $i \in [p]$, the model $q_{k,i}$ satisfies~\eqref{eq:gen_approx_intrpl_cond} on ${\cal D}_k \cup \{0\}$, and $(b)$ the matrix $\begin{bmatrix} d_1 & \cdots & d_{n_x} \end{bmatrix}$ is invertible and, for some $\Lambda \in \RR_{>0}$ and $\bar \kappa_{\apx} \in \RR_{>0}$, \eqref{cond:fl-1} and \eqref{cond:fl-2} hold.  Then, $m_k$ defined by \eqref{eq:master_model} satisfies \eqref{eq:main_thm_fcn_bnd} in $B(x_k, \Delta_k) \cap \Omega$ for some $(\kappa_{\ef},\kappa_{\eg}) \in \RR_{> 0}^2$ independent of $k$.
	
\end{lemma}

%************
% Subsection
%************
\subsection{Algorithm Description} \label{sec.gen_alg_desc}

\begin{algorithm}[t]
	\caption{: DFO Framework with Approx.~Function Values at Interpolation Points}
	\label{alg.dfo_approx}
	\begin{algorithmic}[1]
		\Require $\Delta_0 \in (0,\infty)$; $\Delta_{\max} \in (0,\infty)$; %$\delta_{\max} \in (0,\infty)$; 
		$\gamma_{\dec} \in (0,1)$; $\gamma_{\inc} \in (1,\infty)$; $\eta \in (0,\infty)$; $\mu \in (0,\infty)$; 
		$\epsilon_c \in (0,\infty)$; $\kappa_{\apx} \in (0,\infty)$; $c_{\app} \in [0, \infty)$; 
		$\kappa_{\fcd} \in (0,1]$.
		\Require  Initial iterate $x_0 \in \RR^{n_x}$.
		\For{$k = 0, 1, \dots$}
		\State Set $\delta_{k} \gets c_{\app} \Delta_k^2$. \label{st:delta_Delta_ratio}    
		\State \algparboxtwo{Find $\DDD_{k} \subset \Omega - x_k$ such that~\eqref{cond:fl-1} holds and the approximate values $\{\tilde{F}_i(x_k+d, \delta_k)\}_{i\in[p],d \in \DDD_k\cup\{0\}}$ satisfy \eqref{eq.fec-new} and  yield $\{q_{k,i}\}_{i=1}^p$ and $m_k$ satisfying \eqref{eq:gen_approx_intrpl_cond}, %\eqref{eq.fec-new}, 
			Assumption~\ref{assumption:q}, and \eqref{eq:master_model}. \label{st:intrpl_pts}}
		\If {$\pi_k^m \leq \epsilon_c$ and $\Delta_k > \mu \pi_k^m$}\label{st:check_criticality}  
		\State \underline{Criticality step}: set $s_k=0$, $\rho_k=0$, $\Delta_{k+1} \gets \gamma_{\dec} \Delta_k$, and $x_{k + 1} \gets x_k$. \label{st:criticality}
		\Else
		\State Compute $s_k$ satisfying \eqref{eq:cauchy_dec}. \label{st.optimize_model}
		\State Evaluate $F_i(x_k + s_k)$  for all $i\in [p]$.
		\State Compute $\rho_k$ by \eqref{eq:act_pred_ratio}.
		\If {$\rho_k \geq \eta$} \label{st:cond_1_update}
		\State \underline{Successful step}: set $\Delta_{k+1} \gets \min \{\gamma_{\inc} \Delta_k, \Delta_{\max}\}$ and $x_{k+1} \gets x_k + s_k$.
		\Else \label{st:cond_2_update}
		\State \underline{Unsuccessful step}: set $\Delta_{k+1} \gets \gamma_{\dec} \Delta_k$ and $x_{k+1} \gets x_k$.
		\EndIf
		\EndIf
		\EndFor
	\end{algorithmic}
\end{algorithm}

Our algorithmic framework is stated as Algorithm~\ref{alg.dfo_approx}.  The main structure of the algorithm is similar to the general trust-region DFO framework considered in \cite{hough2021model}. In each iteration $k \in \NN$, a set of directions $\DDD_k$ are determined  such that one obtains approximate function values satisfying
\begin{equation}\label{eq.fec-new}
|F_i(x_k + d) - \tilde F_i(x_k+d, \delta_k) | \leq \kappa_{\apx} \delta_k \ \ \text{for all}\ \ d \in {\cal D}_k \cup \{0\}.
\end{equation}
Since $\delta_k \gets c_{\app} \Delta_k^2$ in line~\ref{st:delta_Delta_ratio}, it is clear that \eqref{cond:fl-2} holds with $ \bar\kappa_{\apx} =c_{\app}\kappa_{\apx}$, which along with \eqref{eq:gen_approx_intrpl_cond} and \eqref{eq:master_model} means that the models $\{q_{k,i}\}_{i\in[p]}$ and $m_k$ satisfy the requirements of the lemmas in the previous section. In the case of solving a sequence of related optimization problems, discussed in Section~\ref{sec.sequence_formulation}, an implementable strategy for approximating function values using prior function value information (to ensure \eqref{eq.fec-new}) is presented and analyzed in Section~\ref{sec:fcn_approx}.

Upon construction of $m_k$, the algorithm considers the stationary measure
\begin{align}\label{eq:stationarity_meas}
\pi_k^m = \Big| \min_{\substack{x_k + d \in \Omega \\ \| d \|_{\tr} \leq 1 }} \nabla m_k(x_k)^T d \Big|.
\end{align}
Specifically, if the algorithm finds that $\pi_k^m$ is smaller than a threshold $\epsilon_c \in \RR_{>0}$ and the trust region radius is greater than $\mu \pi_k^m$ for a constant $\mu \in \RR_{>0}$, then a criticality step (see line~\ref{st:criticality}) is performed, meaning the trust region radius is decreased and the iterate is unchanged.  The purpose of the criticality step is to ensure that the sufficiently small value for $\pi_k^m$ is due to a stationarity measure for $f$ also being sufficiently small (see \eqref{def.pif}), not merely due to model inaccuracy.  In any case, if line \ref{st.optimize_model} is reached, then a step $s_k$ is computed in the trust region that guarantees Cauchy decrease, i.e.,
\begin{align}\label{eq:cauchy_dec}
m_k(x_k) - m_k(x_k + s_k) \geq \kappa_{\fcd} \pi_k^m \min \left \{\frac{\pi_k^m}{ \kappa_{\bhm} + 1}, \Delta_k, 1 \right\}
\end{align}
is achieved for some user-prescribed $\kappa_{\fcd} \in (0,1]$, where $\kappa_{\bhm} \in \RR_{>0}$ is defined as in Lemma~\ref{lemma:bounded_hessian_m}.  Sufficient conditions on $s_k$ to achieve \eqref{eq:cauchy_dec} and an algorithm for obtaining such a step can be found in \cite{conn1993global}; e.g., an exact minimzer of $m_k$ within the trust region and bound constraints satisfies \eqref{eq:cauchy_dec}.  After calculating the step $s_k$, the trial point $x_k + s_k$ is considered.  According to the actual-to-predicted reduction ratio, namely,
\begin{align}\label{eq:act_pred_ratio}
\rho_k \coloneqq \frac{f(x_k) - f(x_k + s_k)}{m_k(x_k) -  m_k(x_k + s_k)},
\end{align}
it is decided whether the trial point should be accepted as the new iterate and how the trust region radius should be updated.  If $\rho_k$ is greater than a threshold $\eta \in \RR_{>0}$ then the trial point is accepted; otherwise, it is rejected.  In addition, if $\rho_k \geq \eta$, then the trust region radius is increased; otherwise, it is decreased.

The main difference between our algorithm and other modern model-based DFO frameworks is that function value approximations may be used at interpolation points in place of exact function values. 
However, if, for some $d \in \DDD_k$, the surrogate model cannot approximate  $F_i(x_k + d)$  with the desired accuracy or there is not enough information available for approximating $F_i(x_k + d)$, then one may need to evaluate the black-box function $F_i$ at $x_k + d$ explicitly.
Another important feature of our framework is that one can also take into account the expense and accuracy of function values when  additional interpolation points are added to the set $\DDD_k$ in line~\ref{st:intrpl_pts}; see Section~\ref{sec.intrplset}.

%*****************
% Algorithm
%*****************

\begin{remark}\label{remark:not-fl}
	\emph{We emphasize that our presentation of Algorithm \ref{alg.dfo_approx}---in particular, its requirement that $\{q_{k,i}\}_{i\in[p]}$ and $m_k$ are fully linear for all $k \in \NN$---has been simplified for our discussion and analysis, even though our analysis could be extended to situations in which fully linear models are not always required.  This is consistent with other modern DFO methods, such as those in \cite{conn2009introduction, hough2021model}, which do not require a fully linear model in every iteration.  For example, as long as $\|\nabla^2 m_k(x_k)\| \leq \kappa_{\bhm}$ and the computed trial step satisfies a Cauchy decrease condition $($see \eqref{eq:cauchy_dec} below$)$ for all $k \in \NN$, one can relax the requirements on $\DDD_k$ in line \ref{st:intrpl_pts} and accept the trial point if $\rho_k \geq \eta$, even if $\|\begin{bmatrix} d_2 & \dots & d_{n_x + 1} \end{bmatrix}^{-1}\|_2 > \frac{\Lambda}{\Delta_k}$ $($i.e., \eqref{cond:fl-1} is violated$)$ or $\delta_k > c_{\apx} \Delta_k^2$ $($i.e., \eqref{cond:fl-2} may be violated$)$.  In such a setting, the algorithm can be modified as follows. 
		%a couple of modifications of the algorithm are needed. 
		First, one needs to modify the step acceptance conditions to stop the algorithm from decreasing the trust region radius if a step has been computed with a model that is not fully linear.  Second, in the criticality step, if the model is not fully linear, then again the trust region radius should not be updated and instead a fully linear model should be constructed.}
\end{remark}

%************
% Subsection
%************
\subsection{Convergence Result} \label{sec.gen_conv_analysis}

A convergence result for Algorithm~\ref{alg.dfo_approx} follows in a similar manner as for Algorithm~10.1 in \cite{conn2009introduction} and Algorithm~1 in  \cite{conejo2013global}.  We state the result in terms of a stationarity measure for the minimization of $f$ over $\Omega$ that is similar to the previously defined measure with respect to $m_k$ (recall \eqref{eq:stationarity_meas}), namely,
\begin{equation}\label{def.pif}
\pi_k^f = \Big| \min_{\substack{x_k + d \in \Omega \\ \| d \|_{\tr} \leq 1 }} \nabla f(x_k)^T d \Big|;
\end{equation}
see \cite{conn1993global, hough2021model}.  Specifically, under the following assumption, the following theorem holds.

\begin{assumption}\label{assumption:bounded_f}
	The objective $f: \RR^{n_x} \to \RR$ is bounded below on $\LLL_{\enl} \cap \Omega$.
\end{assumption}

\begin{theorem}\label{th.main}
	If Assumptions \ref{assumption:F_h}, \ref{assumption:q}, and \ref{assumption:bounded_f} hold, then
	$\displaystyle \lim_{k \to \infty} \pi_k^f = 0$.
\end{theorem}

Given the properties of Algorithm~\ref{alg.dfo_approx} provided in this section, most notably the fact that it generates fully linear models, a complete proof of Theorem~\ref{th.main} follows that of other such convergence analyses of model-based trust-region DFO methods. 
We provide the proof in Appendix~\ref{app:conv_proof} of the supplementary material for completeness.
% As a brief overview of the proof, we provide a sketch in Appendix~\ref{app:conv_proof}.

%*********
% Section
%*********
\section{Interpolation Set Construction}\label{sec.intrplset}

In this section, we provide an implementable method for determining an interpolation set in line~\ref{st:intrpl_pts} of Algorithm~\ref{alg.dfo_approx} that contains $n_x + 1$ sufficiently affinely independent directions in the sense that \eqref{cond:fl-1} holds.  
Our strategy involves first running Algorithm 4.1 in \cite{wild2008mnh}, which is presented here as Algorithm~\ref{alg:Affpoints_main}. It chooses interpolation points from a candidate set $\cal{C}$.
\begin{algorithm}[t]
	\caption{: Construction of elements for $\DDD_k$ from ordered candidate set $\cal{C}$}
	\label{alg:Affpoints_main}
	\begin{algorithmic}[1]
		\Require Values and parameters from Algorithm~\ref{alg.dfo_approx}.%; parameters from Table~\ref{tab.parameters}.
		\State Set $\DDD_k \gets \{0\}$ and $Z \gets I_{n_x}$.
		\State Choose $\CCC \subseteq B_{\tr}(x_k, \Delta_k) \cap \Omega$ as $\{x^{(1)}, \dots, x^{(|\CCC|)}\}$ where $x^{(i)}$ is preferred over $x^{(i+1)}$.\label{st:cal_C_set}
		\For{$j = 1, \dots, |\CCC|$} 
		\If{$\left \|\proj{Z}{\frac{ x^{(j)} - x_k}{\Delta_k}} \right\|_2 \geq \xi$ \label{st:pivot_magnitude_main}}
		\State Set $\DDD_k \gets \DDD_k \cup \{x^{(j)} -  x_k \}$. 
		\State \algparboxthree{Compute orthonormal $Z$ with $\textrm{span}(Z) = \nulll([\DDD_k]_{1:\text{end}})$ (see \eqref{eq.D_not_0}).} \label{st.Z_in_alg}
		\EndIf
		\EndFor
		\State \Return $(\DDD_k, Z)$
	\end{algorithmic}
\end{algorithm}
If the output of this procedure is a set of $n_x + 1$ directions, then the set is complete and we will show that \eqref{cond:fl-1} is guaranteed to hold.  Otherwise, Algorithm~\ref{alg:nextaffpoint_main}, also presented in this section, can be called iteratively to generate additional directions until a total of $n_x + 1$ directions have been obtained, where again we show that \eqref{cond:fl-1} is guaranteed to hold. In either case, after obtaining $n_x + 1$ directions, we may call Algorithm 4.2 in \cite{wild2008mnh} to potentially add additional directions to the interpolation set and construct a quadratic model. Additional interpolation directions have no impact on the conditions required by Lemma~\ref{lemma:comp_fl} and the fully linear nature of models.

Algorithm~\ref{alg:Affpoints_main} determines a set of directions $\DDD_k$ by choosing interpolation points from a given candidate set $\mathcal{C}$.  Typically, $\CCC$ only includes points that are not too far from the current iterate; e.g., they may reside in $B_{\tr}(x_k, \Delta_k) \cap \Omega$.
%and have a utility as measured by a function $u$ that exceeds a threshold $u_{\thresh}$ 
We assume that the points in $\CCC$ are ordered by some preference; see line~\ref{st:cal_C_set}.  Typically, in existing implementations, points are more desirable the closer they are to $x_k$, see \cite{wild2017chapter}.
However, in our context, we may also want to take into account how precise or expensive the function approximation $\tilde F_i$ is at a candidate point.  The specific definition will be problem dependent.  In Section~\ref{sec.u_calE} , we describe a procedure for ordering the elements in $\CCC$ that is tailored to the repeated solution of least-squares problems.

Algorithm~\ref{alg:Affpoints_main} adds elements to $\DDD_k$ iteratively while also maintaining $Z \in \RR^{n_x \times n_z}$ as a corresponding orthonormal matrix such that $\textrm{span}(Z) = \nulll([\DDD_k]_{1:\text{end}})$, where
\begin{equation}\label{eq.D_not_0}
[\DDD_k]_{1:\text{end}} = \text{matrix composed of vectors in $\DDD_k$ except 0}.
\end{equation}
Observe that the norm of the projection in line \ref{st:pivot_magnitude_main} of the algorithm can be computed cheaply since it is equal to $\|Z^T(x^{(i)} - x_k)\|_2/\Delta_k$.

If Algorithm~\ref{alg:Affpoints_main} returns an interpolation set with $|\DDD_k| = n_x + 1$, then \eqref{cond:fl-1} holds (as we show in Theorem~\ref{th:shima}).  Otherwise, additional directions need to be added.  Under the assumption that the (potentially infinite) bound constraints defined by $\Omega = [x_L,x_U]$ are not relaxable, we propose Algorithm~\ref{alg:nextaffpoint_main},
\begin{algorithm}[t]
	\caption{: Construction of additional element for $\DDD_k$}
	\label{alg:nextaffpoint_main}
	\begin{algorithmic}[1]
		\Require $\DDD_k$ and corresponding $Z \in \RR^{n_x \times n_z}$.
		\State Set $d \gets \vecZeros$.
		\For {$i = 1, \dots, n_z$}
		\State $\tau^+ \gets \texttt{\textsc{OptStep}}(Z,z_i,\bar \tau(z_i))$.\label{s:tau_plus}
		\State \textbf{if} $\Upsilon(Z,d(z_i,\tau^+)) > \Upsilon(Z,d)$ \textbf{then} set $d \gets d(z_i,\tau^+)$.
		\State $\tau^- \gets \texttt{\textsc{OptStep}}(Z,-z_i,\bar \tau(-z_i))$.\label{s:tau_minus}
		\State \textbf{if} $\Upsilon(Z,d(-z_i,\tau^-)) > \Upsilon(Z,d)$ \textbf{then} set $d \gets d(-z_i,\tau^-)$.
		\EndFor
		\State Set $\DDD_k \gets \DDD_k \cup \{d\}$. 
		\State Compute orthonormal $Z$ with $\textrm{span}(Z) = \nulll([\DDD_k]_{1:\text{end}})$ (see \eqref{eq.D_not_0}).
		\State \Return $(\DDD_k,Z)$.
	\end{algorithmic}
\end{algorithm} 
which employs the subroutine in Algorithm~\ref{alg:opt_step2_main}, 
\begin{algorithm}[t]
	\caption{: \texttt{OptStep}$(Z, v, \bar \tau)$ (see Algorithm~\ref{alg:nextaffpoint_main})}
	\label{alg:opt_step2_main}
	\begin{algorithmic}[1]
		\State \algparboxtwo{Let $\bar\tau_{(0)} \gets 0$ and sort distinct values of $\{\bar\tau\}_{j=1}^{n_x}$ as $0 < \bar\tau_{(1)} \leq \dots \leq \bar\tau_{(\hat n_x)}$.}
		\If {$\| d(v,\bar\tau_{(\hat n_x)})\|_{\tr}\leq \Delta_k$}
		\State \Return $\argmax_{\tau \in \{\bar\tau_{(j)}\}_{j=1}^{\hat n_x}} \Upsilon(Z, d(v,\tau))$.
		\Else
		\State Find $\hat j \in [n_x]$ such that $\|d(v,\bar\tau_{(\hat j - 1)})\|_{\tr} \leq \Delta_k \leq \| d(v,\bar\tau_{(\hat j )}) \|_{\tr}$.
		\State Find $\hat \tau \in [\bar\tau_{(\hat j - 1)}, \bar\tau_{(\hat j)}]$ such that $\| d(v,\hat\tau) \|_{\tr} = \Delta_k$. \label{st:find_tau_main}
		\State \Return $\argmax_{\tau \in \{\bar\tau_{(j)}\}_{j=1}^{\hat j- 1} \cup \{\hat\tau\}} \Upsilon(Z, d(v,\tau))$.
		\EndIf
	\end{algorithmic}
\end{algorithm}
for augmenting the set $\DDD_k$ in a manner that ensures that \eqref{cond:fl-1} continues to hold.  This algorithm can be called iteratively until $|\DDD_k| = n_x+1$.
\subsection{Generating Interpolation Points for Bound Constraints}\label{sec.intpl_bounds}
Consider $\DDD_k$ with $|\DDD_k| < n_x + 1$ obtained after a call to Algorithm~\ref{alg:Affpoints_main} and (potentially) call(s) to Algorithm~\ref{alg:nextaffpoint_main}.  Since $|\DDD_k| < n_x + 1$, it follows that $\nulll([\DDD_k]_{1:\text{end}}) \neq \{0\}$, so $n_z \geq 1$.  Following the spirit of Algorithm~\ref{alg:Affpoints_main}, one can augment $\DDD_k$ with $d \in \RR^{n_x}$ in a manner that guarantees that \eqref{cond:fl-1} continues to hold by ensuring that
\begin{align}\label{cond1:nextpt_main}
\left \| \proj{Z} {\frac{d}{\Delta_k}} \right\|_2 \geq \xi,\ \ \|d\|_{\tr} \leq \Delta_k,\ \ \text{and}\ \ x_k + d \in \Omega
\end{align}
for $\xi$ sufficiently small; see Theorem~\ref{thm:poised_direction}.  Recalling that $\|\proj{Z} {d/\Delta_k}\|=\|Z^Td\|/\Delta_k$, 
one way in which one might ensure this property is by solving
\begin{align}\label{eq:opt_nextpt_main}
\max_{d \in \RR^n}\ \Upsilon(Z,d)\ \ \st\ \ \| d \|_{\tr} \leq  \Delta_k\ \ \text{and}\ \ x_L - x_k \leq  d \leq x_U - x_k,
\end{align}
where $\Upsilon(Z,d) := \|Z^Td\|_2^2$.  However, a globally optimal solution of \eqref{eq:opt_nextpt_main} may be computationally expensive to obtain, and such a solution is not actually necessary. % to obtain as long as $\xi$ is sufficiently small.

The motivation for our proposed Algorithm~\ref{alg:nextaffpoint_main} is to solve an approximation of problem~\eqref{eq:opt_nextpt_main}, a solution of which turns out to be sufficient for ensuring that \eqref{cond:fl-1} continues to hold for $\xi$ sufficiently small.  In particular, the algorithm solves
\begin{align}\label{eq:opt_subopt_nextpt_main}
\max_{(v,i,\tau) \in \RR^{n_x} \times [n_z] \times \RR_{\geq0}}\ \Upsilon(Z, d(v,\tau))\ \ \st\ \ v \in \{-z_i,z_i\}\ \ \text{and}\ \ \|d(v,\tau)\|_{\tr} \leq \Delta_k,  
\end{align}
where $z_i$ is the $i$th column of $Z$ and where for any $(i,\tau) \in [n_z] \times \RR_{\geq0}$ and $v \in \{-z_i,z_i\}$ the vector $x_k + d(v,\tau)$ is the projection of $x_k + \tau v$ onto $\Omega = [x_L,x_U]$.  A formula for this vector is easily derived.  In particular, for $v \in \RR^{n_x}$, let $\bar \tau(v) \in (\RR \cup \{\infty\})^{n_x}$ have
\begin{align}\label{eq:tau_bar_main}
\bar \tau_j(v) \coloneqq \begin{cases}
\frac{x_{U, j} - x_{k, j}}{v_j} \ \ &\text{if} \ v_j>0 \ \text{and} \ x_{U,j} < \infty,\\
\frac{x_{L, j} - x_{k, j}}{v_j} \ \ &\text{if} \ v_j<0 \ \text{and} \ x_{L,j} > - \infty,\\
\infty \ \ &\text{otherwise}
\end{cases}
\end{align}
for $j\in[n_x]$. The projection of $x_k + \tau v$ onto $[x_L, x_U]$ is given by $x_k + d(v,\tau)$, where
\begin{align}\label{eq:projected_v_2_main}
d(v,\tau) = \begin{bmatrix} \min\{\tau, \bar \tau_1(v)\} v_1 & \dots & \min \{\tau, \bar \tau_{n_x}(v)\} v_{n_x} \end{bmatrix}^T.
\end{align}

\begin{remark}
	\emph{In \eqref{eq:projected_v_2_main} and subsequent calculations below, we interpret equations and operations involving infinite quantities in the following natural ways: $\infty = \infty$; $\infty \times \infty = \infty$; $\infty - (-\infty) = \infty$; $\min \{\infty, \infty\} = \infty$; $b - \infty = -\infty$ and $b + \infty = \infty$ for any $b \in \RR$; $a \times \infty = \infty$ and $-a \times \infty = -\infty$ for any $a \in \RR_{>0}$; and $\|v\| = \infty$ for any norm $\|\cdot\|$ and any $v \in (\RR \cup \{-\infty,\infty\})^{n_x}$ that has an element equal to $-\infty$ or $\infty$. 
		One additional rule that we intend, which is not natural in all contexts, is $0 \times \infty = 0$.} 
\end{remark}

Algorithm~\ref{alg:nextaffpoint_main} operates by iterating through the columns of $Z$, where for each $i \in [n_z]$ the optimal values of $\tau \in \RR_{\geq0}$ (with respect to maximizing $\Upsilon$) are determined along $z_i$ and $-z_i$.  If the search along either direction yields a larger value of $\Upsilon$ than has been observed so far, the solution candidate (with respect to \eqref{eq:opt_subopt_nextpt_main}) is updated.  The subroutine in Algorithm~\ref{alg:opt_step2_main} is responsible for computing the optimal step sizes.

Our goal in the remainder of this section is to prove that the output $\DDD_k$ from Algorithm~\ref{alg:Affpoints_main} satisfies %\eqref{cond:fl-1}
\eqref{cond1:nextpt_main} and, after any subsequent call to Algorithm~\ref{alg:nextaffpoint_main}, the elements of $\DDD_k$ continue to satisfy \eqref{cond1:nextpt_main}
%\eqref{cond:fl-1} 
as long as $\xi$ is sufficiently small.  This notion of sufficiently small is determined by a threshold revealed in Theorem~\ref{thm:poised_direction}.

Our first lemma bounds $\tau \in \RR_{\geq0}$ below if $d(\cdot,\tau)$ lies on the trust-region boundary.

\begin{lemma}\label{lemma:tau_lb}
	If $(v,\tau) \in \RR^{n_x} \times \RR_{\geq0}$, $\|v\|_2=1$, and $\|d(v,\tau)\|_{\tr} = \Delta_k$, then $\tau \geq \frac{\Delta_k}{\kappa_{\tr_1}}$.
\end{lemma}
\begin{proof}
	By \eqref{eq.equiv8}, \eqref{eq:projected_v_2_main}, and the conditions of the lemma, $\tr<\infty$ implies
	\begin{align*}
	\Delta_k^{\tr} = \|d(v,\tau)\|_{\tr}^{\tr} = \sum_{j=1}^{n_x} (\min\{\tau,\bar \tau_j(v)\})^{\tr} |v_j|^{\tr} \leq \tau^{\tr} \|v\|_{\tr}^{\tr} \leq \tau^{\tr} \kappa_{\tr_1}^{\tr} \|v\|_2^{\tr} = \tau^{\tr} \kappa_{\tr_1}^{\tr},
	\end{align*}
	which yields the desired conclusion.  (Here, the superscript ``$\tr$'' denotes the $\tr$-th power of a number.)
	The result for $\tr=\infty$ can be shown in a similar manner.
\end{proof}

Going forward, corresponding to $v \in \RR^{n_x}$, we choose an index $j^*(v)$ such that
\begin{align}\label{eq:big_comp_unitary_vec}
j^*(v) \in \argmax_{j \in [n_x]} |v_j|.  
\end{align}

The following lemma is trivial (recall \eqref{eq:tau_bar_main}), so we state it without proof.

\begin{lemma}\label{lemma:big_comp_idx}
	If $v \in \RR^{n_x}$ has $\|v\|_2 = 1$, then $\frac{1}{\sqrt{n_x}} \leq |v_{j^*(v)} | \leq 1$ and 
	\begin{align*}
	\max\{\bar \tau_{j^*(v)}(-v), \bar \tau_{j^*(v)}(v)\} \geq \frac{x_{U,{j^*(v)}} - x_{L,{j^*(v)}}}{2|v_{j^*(v)}|}.
	\end{align*}
\end{lemma}

Our next lemma shows a lower bound for the optimal value of \eqref{eq:opt_subopt_nextpt_main} if a call to the \texttt{OptStep} subroutine finds all step sizes yielding points within the trust region.

\begin{lemma}\label{lemma:plus_less_Delta}
	Consider line~\ref{s:tau_plus} in Algorithm~\ref{alg:nextaffpoint_main}  for any $i \in [n_z]$. If it is found within the call to \texttt{OptStep}$(Z,z_i,\bar\tau(z_i))$ that $\| d(z_i,\bar\tau_{(\hat n_x)}(z_i)) \|_{\tr} \leq \Delta_k$, then $\tau^+ \in \RR_{\geq0}$ yields $\Upsilon(Z,d(z_i,\tau^+)) \geq \frac{(\bar \tau_{j^*(z_i)}(z_i))^2}{n_x^2}$.  The same property holds for line~\ref{s:tau_minus} and \texttt{OptStep}$(Z,-z_i,\bar\tau(-z_i))$ with respect to $\tau^- \in \RR_{\geq0}$.
\end{lemma}
\begin{proof}
	Without loss of generality, consider the call to \texttt{OptStep}$(Z,z_i,\bar\tau(z_i))$.  The proof for the call to \texttt{OptStep}$(Z,-z_i,\bar\tau(-z_i))$ is nearly identical.  By the conditions of the lemma, \eqref{eq:projected_v_2_main}, Lemma \ref{lemma:big_comp_idx}, and the fact that $\bar\tau_{(\hat n_x)} \geq \bar \tau_{j^*(z_i)}(z_i)$, it follows that
	\begin{align*}
	\Upsilon(Z,d(z_i,\tau^+))
	\geq&\ \Upsilon(Z,d(z_i,\bar\tau_{(\hat n_x)})) = \sum_{l=1}^{n_z} (z_l^Td(z_i,\bar\tau_{(\hat n_x)}))^2 \\
	\geq&\ \left(\sum_{j=1}^{n_x} [z_i]_j d_j(z_i,\bar\tau_{(\hat n_x)})\right)^2 = \left(\sum_{j=1}^{n_x} [z_i]_j^2 \min\{\bar\tau_{(\hat n_x)}, \bar \tau_j(z_i)\} \right)^2 \\
	\geq&\ [z_i]_{j^*(z_i)}^4 \min \{\bar\tau_{(\hat n_x)}, \bar \tau_{j^*(z_i)}(z_i)\}^2 \geq \frac{\bar \tau_{j^*(z_i)}(z_i)^2}{n_x^2},
	\end{align*}
	as desired.
\end{proof}

Next, we present the following lemma, which considers cases when a call to the \texttt{OptStep} subroutine finds that not all step sizes yield points within the trust region.

\begin{lemma}\label{lemma:plus_greater_Delta} Consider line~\ref{s:tau_plus} in Algorithm~\ref{alg:nextaffpoint_main} for any $i \in [n_z]$. If it is found in the call to \texttt{OptStep}$(Z,z_i,\bar\tau(z_i))$  that $\| d(z_i,\bar\tau_{(\hat n_x)}(z_i)) \|_{\tr} > \Delta_k$, then $\tau^+ \in \RR_{\geq0}$ yields
	\begin{equation*}
	\Upsilon(Z,d(z_i,\tau^+)) \geq \left( \min \left\{ \frac{\Delta_k}{\kappa_{\tr_1}}, \bar \tau_{j^*(z_i)}(z_i) \right\} \right)^2 \frac{1}{n_x^2}.
	\end{equation*}
	The same property holds for  line~\ref{s:tau_minus} and \texttt{OptStep}$(Z,-z_i,\bar\tau(-z_i))$ w.r.t.~$\tau^- \in \RR_{\geq0}$.
\end{lemma}
\begin{proof}
	The proof is nearly identical to that of Lemma~\ref{lemma:plus_less_Delta}, except with the computed value $\hat\tau$ in place of $\bar\tau_{(n_x)}$, where by Lemma~\ref{lemma:tau_lb} one has that $\hat\tau \geq \frac{\Delta_k}{\kappa_{\tr_1}}$.
\end{proof}

We now prove our main result of this section.

\begin{theorem}\label{thm:poised_direction}
	If the well-poisedness parameter is chosen such that
	\begin{equation*}
	\xi \in \left(0, \frac{1}{n_x} \min \left\{\frac{1}{\kappa_{\tr_1}}, \frac{\min_{j \in [n_x]}(x_{U, j} - x_{L, j})}{2\Delta_{\max}} \right\} \right],
	\end{equation*}
	then any call to Algorithm \ref{alg:nextaffpoint_main} adds $d \in \RR^{n_x}$ to $\DDD_k$ that satisfies \eqref{cond1:nextpt_main}.
\end{theorem}
\begin{proof}
	By construction, the vector $d \in \RR^{n_x}$ that is added to the elements of $\DDD_k$ by the algorithm satisfies $\|d\|_{\tr} \leq \Delta_k$ and $x_k + d \in \Omega = [x_L,x_U]$.  Hence, all that remains is to prove that $d$ also satisfies the first inequality in \eqref{cond1:nextpt_main}.  According to the construction of the algorithm, it is sufficient to show that for some $i \in [n_z]$ one finds
	\begin{align*}
	\frac{1}{\Delta_k} \Upsilon_{i,\max} := \frac{1}{\Delta_k} \max \left\{\sqrt{\Upsilon(Z,d(z_i,\tau^+))}, \sqrt{\Upsilon(Z,d(-z_i,\tau^-))} \right\} \geq \xi.
	\end{align*}
	If the calls to \texttt{OptStep}$(Z,z_i,\bar\tau(z_i))$ and \texttt{OptStep}$(Z,-z_i,\bar\tau(-z_i))$ both find that all step sizes yield points within the trust region, then Lemmas~\ref{lemma:big_comp_idx} and \ref{lemma:plus_less_Delta} imply
	\begin{align*}
	\frac{1}{\Delta_k} \Upsilon_{i,\max}
	\geq&\ \frac{1}{n_x \Delta_k} \max \{\bar \tau_{j^*(z_i)}(z_i), \bar \tau_{j^*(-z_i)}(-z_i) \} \\
	\geq&\ \frac{1}{n_x \Delta_k} \frac{x_{U, j^*(z_i)} - x_{L, j^*(z_i)}}{2 |[z_i]_{j^*(z_i)}|} \geq \frac{x_{U, j^*(z_i)} - x_{L, j^*(z_i)}}{2 n_x \Delta_{\max}}.
	\end{align*}
	Otherwise, if the call to \texttt{OptStep}$(Z,z_i,\bar\tau(z_i))$ and/or \texttt{OptStep}$(Z,-z_i,\bar\tau(-z_i))$ finds that some of the step sizes yield points that are outside of the trust region, then one may conclude from Lemma~\ref{lemma:plus_greater_Delta} that
	\begin{align*}
	\frac{1}{\Delta_k} \Upsilon_{i,\max} \geq \frac{1}{n_x \Delta_k} \max \left\{ \min \left\{ \frac{\Delta_k}{\kappa_{\tr_1}}, \bar \tau_{j^*(z_i)}(z_i)\right\},  \min \left\{ \frac{\Delta_k}{\kappa_{\tr_1}}, \bar \tau_{j^*(-z_i)}(-z_i)\right\} \right\}.
	\end{align*}
	Considering all possible cases for which term obtains in the minima in the expression above, one may conclude with Lemmas~\ref{lemma:tau_lb} and \ref{lemma:big_comp_idx} that
	\begin{align*}
	\frac{1}{\Delta_k} \Upsilon_{i,\max} \geq \frac{1}{n_x} \min \left\{ \frac{1}{ \kappa_{\tr_1}}, \frac{x_{U, j^*(z_i)} - x_{L, j^*(z_i)}}{2 \Delta_{\max}} \right\}.
	\end{align*}
	Combining the results of these cases, the desired conclusion follows.
\end{proof}

The next lemma shows that the norm in line~\ref{st:pivot_magnitude_main} of Algorithm~\ref{alg:Affpoints_main} and the first inequality in \eqref{cond1:nextpt_main} measures the magnitude of a pivot of a QR factorization.

\begin{lemma}\label{lemma:proj_norm_QR}
	%   Suppose line~\ref{st:pivot_magnitude_main} of Algorithm~\ref{alg:Affpoints_main} has been reached and 
	At line~\ref{st:pivot_magnitude_main} of Algorithm~\ref{alg:Affpoints_main}, one finds with $d = (x^{(i)} - x_k)$ that
	\begin{align*}
	\left \|\proj{Z}{\frac{d}{\Delta_k}} \right\|_2 = |r|,
	\end{align*}
	where $r$ is the last diagonal of $R$ in a QR factorization of $\frac{1}{\Delta_k} \begin{bmatrix} [\DDD_k]_{1:{\rm end}} & d \end{bmatrix}$.
\end{lemma}
\begin{proof}
	Letting $QR$ be a QR factorization of $\frac{1}{\Delta_k} [\DDD_k]_{1:{\rm end}}$, one finds that a QR factorization of the augmented matrix has, for some vector $q$ satisfying $Q^Tq = 0$,
	\begin{align*}
	\frac{1}{\Delta_k} \begin{bmatrix} [\DDD_k]_{1:{\rm end}} & d \end{bmatrix} = \begin{bmatrix} Q & q \end{bmatrix} \begin{bmatrix} R & v \\ \vecZeros^T & r \end{bmatrix}.
	\end{align*}
	Left-multiplication by $Z^T$ yields $Z^T \frac{d}{\Delta_k} = (Z^Tq)r$.  Then, since $Q^Tq = 0$, it follows that $q = Zu$ for some vector $u$ with $\|u\|_2 = 1$, from which it follows that
	\begin{align*}
	\left \|\proj{Z}{\frac{d}{\Delta_k}} \right\|_2 = \left\|Z^T \left( \frac{d}{\Delta_k} \right) \right \|_2 = |r| \left\|Z^T q \right \|_2 = |r| \left \| Z^T Z u \right \|_2 = |r|,
	\end{align*}
	as desired.
\end{proof}

The following theorem is similar to Lemma~4.2 in \cite{wild2008orbit} and Lemma~3.2 in \cite{stefan2008derivative}.  That said, we include its proof in Appendix~\ref{app:lemmas_proof} of the supplementary material for completeness. 

\begin{theorem}\label{th:shima}
	Once Algorithm~\ref{alg:Affpoints_main} or iterative calls to Algorithm~\ref{alg:nextaffpoint_main} yields an interpolation set $\DDD_k$ with $|\DDD_k| = n_x+1$, it follows that \eqref{cond:fl-1} holds with 
	$\displaystyle
	\Lambda = \frac{n_x^{\frac{n_x - 1}{2}} \kappa_{\tr_0}^{n_x - 1}}{ \xi^{n_x}}.
	$
\end{theorem}

% Theorem~\ref{th:shima} shows that if Algorithm~\ref{alg:Affpoints_main} yields $|\DDD_k| = n_x+1$, then \eqref{cond:fl-1} holds.  

%***********
% Section
%***********
\section{Function Approximation Using Regression} \label{sec:fcn_approx} 

Next, we provide a function approximation scheme that is tailored to the setting of Section~\ref{sec.sequence_formulation}, where a sequence of optimization problems is solved. 
We describe a simple regression procedure to compute an estimate $\tilde{F}_i(x,\ttt,\delta)$ of $F_i(x,\theta)$ that satisfies the approximation condition \eqref{eq:gen_err_ub_0_seq}, from prior evaluations of $\tilde F_i$ near $(x,\theta)$.
In Section~\ref{sec.multioutput}, we first discuss multi-output simulation-based optimization, then in Section~\ref{sec.leastsquares} we consider least-squares regression problems.

For completeness, we also describe in Section~\ref{sec.u_calE} our choice of $\cal{C}$ for Algorithm~\ref{alg:Affpoints_main} in the least-squares case.

%*************
% subsection
%*************
\subsection{Multi-Output Simulation-based Optimization} \label{sec.multioutput}
Recall the multi-output simulation-based optimization setting with $\HHH_i$ defined in \eqref{eq:history_multi_output} and consider arbitrary $i \in [p]$.
In this section, we make the following assumption. % about the function $F_i$. 
{
	Here, $\|\cdot\|_{\app}$ is used to define the approximation radius (i.e., precision parameter) in our algorithm for some real number $\app \geq 1$.}
\begin{assumption}  
	\label{assumption:F_i_lipschitz}
	There exists an open convex set $\Xi$ containing $\LLL_{\enl} \times \Theta$ over which $F_i$ is Lipschitz continuous with constant $L_{F_i} \in \RR_{>0}$ such that 
	\begin{align*}
	\left|F_i(x, \ttt) - F_i(\bar x, \bar \ttt) \right| \leq L_{F_i} \left \| \begin{bmatrix}
	x - \bar x \\
	\ttt - \bar \ttt
	\end{bmatrix}\right \|_{\apx}
	\end{align*}
	for all $\left((x, \ttt), (\bar x, \bar \ttt) \right) \in \Xi \times \Xi$.
\end{assumption}
Let $(x,\theta)\in\Xi$ and $\delta\in\RR_{>0}$.
Estimating $\tilde{F}_i(x,\ttt,\delta) \approx F_i(x,\theta)$ requires some prior function evaluations near $(x,\ttt)$.
We capture such information in a set that we define 
\begin{align}\label{eq:available_pts_multi_output}
\AAA_i(x, \ttt, \delta) = \left\{ (\bar{x}, \bar \ttt, F_i(\bar x, \bar \ttt)) \in \HHH_i : \left\| \begin{bmatrix} \bar{x} - x \\ \bar \ttt - \ttt \end{bmatrix}\right \|_{\app} \leq  \delta \right\}.
\end{align}
For notational convenience, let us explicitly express the finite set $\AAA_i(x, \ttt,\delta) = \{(\xttt_j,\phi_j)\}_{j=1}^N$, where $\xttt_j := (x_j^T,\ttt_j^T)^T \in \RR^{n_{\xttt}}$ and $\phi_j = F_i(x_j, \ttt_j)$ for all $j \in [N]$.  Given $\AAA_i(x, \ttt,\delta)$, we compute a linear model approximately, namely,
\begin{align*}
\tilde{F}_i(x,\ttt,\delta) = \alpha_0 + \alpha_1^T x + \alpha_2^T \ttt
\end{align*}
of $F_i(x,\ttt)$ with parameters $\alpha^T = \begin{bmatrix} \alpha_0 & \alpha_1^T & \alpha_2^T \end{bmatrix}$ that is computed by solving a regularized regression problem with regularization parameter $\lambda \in \RR_{>0}$, namely,
\begin{align}
\label{eq:opt_regularized_reg_1}
\min_{\alpha \in \RR^{n_{\xttt} + 1}} \frac{1}{2} \sum_{j=1}^N \left(\alpha_0 + \alpha_1^T x_j + \alpha_2^T \ttt_j  - \phi_j \right)^2 + \frac{\lambda}{2} \left(\|\alpha_1\|_2^2 + \|\alpha_2\|_2^2 \right).
\end{align}
Observe that the intercept $\alpha_0$ is not regularized in this problem, which has the result that the approximation is invariant to translations (see Lemma~\ref{lemma:reg_shifting_scaling} below).  For further discussion about regularizing in this manner, see \cite[Section 3.4.1]{hastie2008elements}.

For the computation of the approximation, we define $\xttt := (x^T, \ttt^T)^T$ along with
\begin{align}
M \coloneqq \begin{bmatrix} 1 & \dots & 1 \\ \xttt_1 & \dots & \xttt_N \end{bmatrix}^T \in \RR^{ N \times (n_\xttt+1)}\ \ \text{and}\ \ \phi \coloneqq \begin{bmatrix} \phi_1 & \dots & \phi_N \end{bmatrix}^T \in \RR^{N}. \label{eq:vec_phi}  
\end{align}
In particular, expressing \eqref{eq:opt_regularized_reg_1} in terms of $(M,\phi)$ yields
\begin{equation}
\label{eq:opt_regularized_reg_2}
\min_{\alpha \in \RR^{n_{\xttt}+1}}\ \frac{1}{2} \left\|M \alpha - \phi\right\|_2^2 +\frac{\lambda}{2} \left(\|\alpha_1\|_2^2 + \|\alpha_2\|_2^2 \right).
\end{equation}
Defining $\bar I := I - \unitVec_1 \unitVec_1^T$, one finds that the unique solution of this problem is given by $\alpha = \left(M^TM + \lambda \bar I \right)^{-1}M^T \phi$, from which it follows that the desired approximation is obtained by solving \eqref{eq:opt_regularized_reg_2}, then computing
\begin{align}\label{eq:reg_approx}
\tilde F_i(x, \ttt, \delta) :=   \begin{bmatrix} 1 & \xttt^T \end{bmatrix} \alpha = \beta^T \phi \approx F_i(x, \ttt),
\end{align}
where for the given regularization parameter $\lambda \in \RR_{>0}$ we define
\begin{align}\label{eq:vec_beta}
\beta^T = \begin{bmatrix} 1 & \xttt^T \end{bmatrix} (M^T M + \lambda \bar I)^{-1} M^T.
\end{align}

Our goal in the remainder of this section is to show that \eqref{eq:reg_approx} satisfies the requirements of Algorithm~\ref{alg.dfo_approx}, namely, a bound of the form in \eqref{eq.fec-new}.  Our first lemma shows that the value of the vector $\beta$ defined in \eqref{eq:vec_beta} has invariance properties.

\begin{lemma}\label{lemma:reg_shifting_scaling}
	The vector $\beta$ defined in \eqref{eq:vec_beta} is invariant to translations in the sense that for any vector $\xttt_c \in \RR^{n_{\xttt}}$ one finds that
	\begin{align*}
	\begin{bmatrix} 1 & (\xttt - \xttt_c)^T \end{bmatrix} \left(M_c^T M_c + \lambda \bar I \right)^{-1} M_c^T = \begin{bmatrix} 1 & \xttt^T \end{bmatrix} (M^T M + \lambda \bar I)^{-1} M^T,
	\end{align*}
	where
	\begin{align*}
	M_c \coloneqq \begin{bmatrix} 1 & \dots & 1 \\ \xttt_1 - \xttt_c & \dots & \xttt_N - \xttt_c \end{bmatrix}^T.
	\end{align*}
\end{lemma}
\begin{proof}
	Defining $C :=  \bar I + \unitVec_1 \begin{bmatrix} 1 & -\xttt_c^T \end{bmatrix}$,  one finds that $M_c = MC$; thus,
	\begin{align*}
	&\ \begin{bmatrix} 1 & (\xttt -\xttt_c)^T \end{bmatrix} \left(M_c^T M_c + \lambda \bar I \right)^{-1} M_c^T \\
	=&\ \begin{bmatrix} 1 & \xttt^T \end{bmatrix} C (C^T (M^T M + \lambda \bar{I}) C)^{-1} C^TM^T = \begin{bmatrix} 1 & \xttt^T \end{bmatrix} (M^T M + \lambda \bar I)^{-1} M^T,
	\end{align*}
	where the first equation follows since $\bar{I} = C^T \bar{I} C$.
\end{proof}

Our next lemma, which is similar to Lemma~4.9 in \cite{conn2009introduction}, shows that \eqref{eq:reg_approx} yields an affine combination of the elements of $\phi$ since the elements of $\beta$ in \eqref{eq:vec_beta} sum to 1.
\begin{lemma}\label{lemma:reg_affine_comb}
	The vector $\beta$ defined in \eqref{eq:vec_beta} satisfies $\beta^T \vecOnes = 1$.
\end{lemma}
\begin{proof}
	Let $\xttt_c = u$, and $M_c = M - \vecOnes\begin{bmatrix} 0 & \xttt^T \end{bmatrix}$.  By Lemma \ref{lemma:reg_shifting_scaling}, one finds
	\begin{align*}
	\beta^T = \begin{bmatrix} 1 & \xttt^T \end{bmatrix} (M^T M + \lambda \bar I)^{-1} M^T = \unitVec_1^T (M_c^T M_c + \lambda \bar I)^{-1} M_c^T.  
	\end{align*}
	Hence, the result follows as long as $\unitVec_1^T (M_c^T M_c + \lambda \bar{I})^{-1} M_c^T \vecOnes = 1$.  To see this, observe that since the first column of $M_c$ is $\vecOnes$, one finds $(M_c^T M_c + \lambda \bar I) \unitVec_1 = M_c^T \vecOnes$.  Multiplying both sides on the left by $(M_c^T M_c + \lambda \bar I)^{-1}$, one concludes that $\unitVec_1 = (M_c^T M_c + \lambda \bar I)^{-1} M_c^T \vecOnes$.  Thus, $\unitVec_1^T (M_c^T M_c + \lambda \bar I)^{-1} M_c^T \vecOnes = \unitVec_1^T \unitVec_1 = 1$, as desired.
\end{proof}
We are now prepared to prove our main result of this section. 
\begin{theorem}
	% \label{lemma:bounded_reg}
	\label{lemma:bounded_reg_multioutput}
	Let $i\in[p]$, $(x,\theta)\in\Xi$ and $\delta\in\RR_{>0}$, and suppose that $\AAA_i(x, \ttt, \delta)\neq\emptyset$. Letting $\sigma_{\min} \in \RR_{\geq0}$ be the minimum eigenvalue of $M_d^T M_d$, where $M_d \coloneqq \begin{bmatrix} \xttt_1 - \xttt_j & \dots & \xttt_N - \xttt_j \end{bmatrix}^T$, it follows that $\tilde F_i(x, \ttt,  \delta)$ defined in \eqref{eq:reg_approx} satisfies
	\begin{align}
	\label{eq:reg_bound}
	\left |\tilde F_i(x, \ttt,  \delta) - F_i(x, \ttt)\right| \leq\ \kappa_{\app} \delta  \  \text{where} \ \kappa_{\app} := L_{F_i} \left(1 + \left(\frac{2N \kappa_{\app_0^*}}{\sigma_{\min} + \lambda} \right) \delta \right). 
	\end{align}
\end{theorem}
\begin{proof}
	By Lemma \ref{lemma:reg_affine_comb}, it follows that $\beta^T \vecOnes = 1$, meaning
	\begin{align}
	|\tilde F_i(x, \ttt, \delta) - F_i(x, \ttt)|
	&= |\beta^T\phi -  F_i(x, \ttt) \beta^T\vecOnes | \leq \sum_{j=1}^N |\beta_j| |\phi_j - F_i(x, \ttt)| \nonumber \\
	& = \sum_{j=1}^N |\beta_j| |F_i(x_j, \ttt_j) - F_i(x, \ttt)| \leq L_{F_i} \| \beta \|_1  \left \| \begin{bmatrix}
	x_j - x \\
	\ttt_j - \ttt
	\end{bmatrix}\right \|_{\apx}  \nonumber\\
	&\leq L_{F_i} \|\beta\|_1 \delta, \label{eq:approx_ub1}
	\end{align}
	where the last inequality follows by the definition of $\AAA_i(x,\ttt,\delta)$.  Our aim now is to bound $\| \beta \|_1$.  Letting $\xttt_c = \frac{1}{N} \sum_{j=1}^N \xttt_j$ and $M_c = M - \vecOnes \begin{bmatrix} 0 & \xttt_c^T \end{bmatrix}$, one finds that
	\begin{align*}
	(M_c^T M_c + \lambda \bar{I})^{-1} = \begin{bmatrix} \frac{1}{N} & \vecZeros^T \\
	\vecZeros & (M_d^TM_d + \lambda I)^{-1} \end{bmatrix}.
	\end{align*}
	Therefore, by applying the result of Lemma~\ref{lemma:reg_shifting_scaling}, one finds that
	\begin{align*}
	\beta^T = \begin{bmatrix} 1 & (\xttt -\xttt_c)^T \end{bmatrix} \left( M_c^T M_c + \lambda \bar{I} \right)^{-1} M_c^T = \tfrac{1}{N} \vecOnes^T + (\xttt -\xttt_c)^T (M_d^TM_d + \lambda I)^{-1} M_d^T;
	\end{align*}
	thus, for any $j \in [N]$, the $j$th element of $\beta$ satisfies
	\begin{align*} 
	|\beta_j|
	&= \left| \frac{1}{N} + (\xttt -\xttt_c)^T (M_d^T M_d + \lambda I)^{-1} (\xttt_j - \xttt) \right| \\
	&\leq \frac{1}{N} + \|\xttt -\xttt_c\|_{\app} \|(M_d^T M_d + \lambda I)^{-1} (\xttt_j - \xttt) \|_{\app^*} \\
	&\leq \frac{1}{N} + \|(M_d^T M_d + \lambda I)^{-1} (\xttt_j - \xttt) \|_{\app^*} \delta \\
	&\leq \frac{1}{N} + \kappa_{\app_0^*} \|\xttt_j - \xttt\|_{\app} \|(M_d^T M_d + \lambda I)^{-1}  \|_{(2,2)} \delta
	= \frac{1}{N} + \left(\frac{2\kappa_{\app_0^*}}{\sigma_{\min} + \lambda} \right) \delta^2.
	\end{align*}
	Therefore, it follows that
	\begin{align*}
	\|\beta\|_1 = \sum_{j=1}^{N} |\beta_j| \leq 1 + \left(\frac{2 N \kappa_{\app_0^*}}{\sigma_{\min} + \lambda}\right) \delta^2,
	\end{align*}
	which combined with \eqref{eq:approx_ub1} yields the desired conclusion.
\end{proof}
\Footnote{AW: We could add a comment here along the lines of: Note that the conclusion of Theorem~\ref{lemma:bounded_reg_multioutput} holds when $\sigma_\min=0$.  If one may want to reduce the approximation error $\kappa_{\app}$, some elements in $\cal A_i$ could be dropped}

%************
% Subsection
%************
\subsection{Least-Squares Optimization Problems}
\label{sec.leastsquares}
Recall problem \eqref{eq:opt_least_square} and let us make the following assumption.
\begin{assumption}
	\label{assumption:phi_lipschitz}
	There exists an open convex set $\Xi$ containing $\LLL_{\enl} \times \cal W$ over which $\phi$ is Lipschitz continuous with constant $L_{\phi} \in \RR_{>0}$ such that
	\begin{align*}
	|\phi(x, w) - \phi(\bar x, \bar w) | \leq L_{\phi} \left \| \begin{bmatrix}
	x - \bar x \\
	w - \bar w
	\end{bmatrix}\right \|_{\apx}
	\end{align*}
	for all $\left((x, w), (\bar x, \bar w) \right) \in \Xi \times \Xi$.
\end{assumption}

The procedure for approximating $F_i(x, \ttt) = \phi(x, w_i)$ for some $\delta \in \RR_{>0}$ is identical to the multi-output case, except that the set with relevant information is defined by
\begin{align}\label{eq:available_pts}
\AAA(x, w_i, \delta) = \left\{ (\bar{x}, \bar w, \phi(\bar x, \bar w)) \in \HHH_i : \left\| \begin{bmatrix} \bar{x} - x \\ \bar w - w_i \end{bmatrix}\right \|_{\app} \leq  \delta \right\},
\end{align}
due to the different definition of prior information in \eqref{eq:history_least_square}.
Theorem~\ref{lemma:bounded_reg_multioutput} still holds, except that $L_{F_i}$ is replaced by $L_{\phi}$.

%************
% Subsection
%************
\subsection{Choice of the candidate set $\cal C$}
\label{sec.u_calE}
In this section, we describe the choice of the candidate set $\cal C$ in Algorithm~\ref{alg:Affpoints_main} in the situation where a sequence of least-squares problems is solved.  We used this procedure in our numerical experiments.
First, we consider an initial set $\mathcal{C}_{\text{init}}$ that contains all feasible points in the history $\HHH_{k, i}$ (defined in \eqref{eq:history_least_square}) and the trust region, namely, 
\begin{align}
\label{eq.tmp_set}
\mathcal{C}_{\text{init}} := \left \{x: (x, w, \phi(x, w)) \in \HHH_{k, i}\right \} \cap B_{\tr}(x_k, \Delta_k) \cap \Omega.
\end{align}
This set contains the interpolation points for which some prior function evaluation, for some $w$, has been computed before.  Next, for each $x \in \CCC_{\text{init}}$, we calculate the score
\begin{align*}
u(x) = \sum_{i=1}^p \mathbb{I}_{\{|\AAA(x, w_i, \delta_k)| \geq 1 \}},
\end{align*}
where $\mathbb{I}_{\{|\AAA(x, w_i, \delta_k)| \geq 1 \}}$ equals one if $|\AAA(x, w_i, \delta_k)| \geq 1$ and zero otherwise. In other words, for each $x \in \mathcal{C}_{\text{init}}$, the score $u(x)$ counts the number of element functions $F_i(\cdot, \theta) = \phi(\cdot, w_i)$ that can be approximated by regularized regression at $x$ instead of being evaluated exactly.  These scores are used to define the unsorted candidate set
\begin{align*}
\mathcal{C}_{\thresh} := \{x \in \mathcal{C}_{\text{init}} : u(x) \geq u_{\thresh}\}.
\end{align*}
Finally, the set $\cal C$ in line~\ref{st:cal_C_set} of Algorithm~\ref{alg:Affpoints_main} is obtained by sorting the elements of $\mathcal{C}_{\thresh}$ according to increasing Euclidean distance to $x_k$.

%*********
% Section
%*********
\section{Numerical Experiments}
\label{sec.numerical} 

The purpose of our experiments is to demonstrate the reduction in function evaluations that can be achieved by our method through its exploitation of prior function evaluations when solving a sequence of related problems.  We compare Algorithm~\ref{alg.dfo_approx} as it is stated, referred to in this section as $\alg_{\HHH}$, and an algorithm that has all of the same features of Algorithm~\ref{alg.dfo_approx} except that does not utilize prior function evaluations to \emph{approximate} function values at interpolation points, referred to as $\alg_{\emptyset}$. Hence, similar to existing model-based DFO algorithms, $\alg_{\emptyset}$ utilizes function evaluations that are obtained in prior iterations by possibly choosing them as candidate interpolation points for Algorithm~\ref{alg:Affpoints_main} in subsequent iterations, but it does not use such information for approximation.  We set $\tr = 2$ and $\app = 2$ in our experiments. We choose $T = 100$, and for all $t \in \{0, \dots, T-1\}$ we run both algorithms until a fixed budget of two simplex gradient evaluations (i.e., $2p(n_x + 1)$ function evaluations) is exhausted.  
We consider such a relatively small amount of function evaluations because the advantage of our method is exploited best when solution time is crucial, such as an online setting, in which the true optimum is out of reach and a good approximate solution suffices.

We implemented our algorithm in Python  by using tools from the Scipy and Numpy packages. We ran our experiments on a Linux workstation  with the following hardware specifications: 2 Intel Xeon Gold 5218R CPU with 20 cores @ 2.10GHz, and 256GB RAM. The use of prior function values obtained during the solution of previous instances is considered as an option in our implementation. To run $\alg_{\emptyset}$, this option is turned off.   We tested our algorithm on a variety of problems,  {including two problems from COPS problem set \cite{dolan2004benchmarking}, namely, Methanol to Hydrocarbons and Catalytic Cracking of Gas Oil. The results were similar in all cases.} 
%For our purposes 
Here, we present the results obtained from a single representative least-squares problem involving an ODE that describes the conversion of methanol into various hydrocarbons \cite{dolan2004benchmarking}, which for parameters $x = \begin{bmatrix} x_1 & \dots & x_5 \end{bmatrix}^T \in \RR^5_{\geq 0}$ and state $v(\tau) = \begin{bmatrix} v_1(\tau) & v_2(\tau) & v_3(\tau) \end{bmatrix}^T \in \RR^3$ is described (with $\tau$ denoting time) by:
\begin{align*}
\frac{d v_1}{d \tau} &= - \Big(2x_2 - \frac{x_1 v_2}{(x_2 + x_5) v_1 + v_2} + x_3 + x_4\Big) v_1; \\
\frac{d v_2}{d \tau} &= \frac{x_1 v_1 (x_2 v_1 - v_2)}{(x_2 + x_5) v_1 + v_2} + x_3 v_1; \\
\frac{d v_3}{d \tau} &= \frac{x_1 v_1(v_2 + x_5 v_1)}{(x_2 + x_5)v_1 + v_2} + x_4 v_1.
\end{align*}
Here, the constraint set is $\Omega_t = \RR^5_{\geq0}$ for all $t \in \{0,\dots,T-1\}$.

We set $\xi = 10^{-3}$ which satisfies the bound in Theorem~\ref{thm:poised_direction}. In practice, $\xi$ should not be too small, nor should it be close to the  upper bound provided in Theorem~\ref{thm:poised_direction} as the system of equations for obtaining model parameters might become ill-conditioned or the requirement for adding the next interpolation point might become too restrictive, respectively. Moreover, the regularization parameter $\lambda$  ($= 10^{-6}$ in our experiments) in \eqref{eq:opt_regularized_reg_1}  is typically a small number since its role is not to decrease the norm of optimal $\alpha_1$ and $\alpha_2$, but merely to make \eqref{eq:vec_beta} have a unique solution, even if $M^T M$ is ill-conditioned.

We fix a vector $\bar{x} = \begin{bmatrix}
1.78 & 2.17 & 1.86 & 1.80 & 0
\end{bmatrix}^T \in \RR^5_{\geq0}$,
%\Footnote{AW: Need to say what this is in our experiments.  Maybe just give the specific values?}
then, iteratively for each $t \in \{0,\dots,T-1\}$, we generate the data for the $t$th least-squares objective in the following manner.  First, we establish seven initial conditions by taking each of the following vectors, perturbing it by a realization of a random vector having a uniform distribution over a 2-norm ball with radius 0.1, then projecting the result onto the 3-dimensional standard simplex so that the elements are nonnegative and sum to one:
\begin{equation*}
\left\{
\begin{bmatrix} 1 \\ 0 \\ 0 \end{bmatrix}, 
\begin{bmatrix} \tfrac34 \\ \tfrac14 \\ 0 \end{bmatrix}, 
\begin{bmatrix} \tfrac34 \\ 0 \\ \tfrac14 \end{bmatrix}, 
\begin{bmatrix} \tfrac12 \\ \tfrac12 \\ 0 \end{bmatrix}, 
\begin{bmatrix} \tfrac12 \\ 0 \\ \tfrac12 \end{bmatrix}, 
\begin{bmatrix} \tfrac14 \\ \tfrac34 \\ 0 \end{bmatrix}, 
\begin{bmatrix} \tfrac14 \\ 0 \\ \tfrac34 \end{bmatrix}
\right\}.
\end{equation*}
(This projection of the initial condition is meaningful for the application since the state elements correspond to proportions.)  We denote the resulting $l$th initial condition by $v_{(l)}^0 \in \RR^3_{\geq0}$ for all $l \in [7]$.  Second, we establish the time points $\tau_{(1)} = 0.1$, $\tau_{(2)} =0.4$, and $\tau_{(3)} =0.8$, which are fixed for all $t \in \{0,\dots,T-1\}$.  At this point in the data generation for problem $t$, we have established $\{w_{i,t}\}_{i\in[21]}$, with each one corresponding to a given initial condition and time point; specifically, each $i \in [21]$ corresponds to a unique pair $(j,l) \in [3] \times [7]$, corresponding to which we define $w_{i,t} := \begin{bmatrix} \tau_{(j)} & (v_{(l)}^0)^T \end{bmatrix}^T$.  All that remains to generate the data for problem $t$ is to establish the values $\{y_{i,t}\}_{i\in[21]}$.  For this, we first generate $x_{(t)} \in \RR^5_{\geq0}$ by adding to $\bar{x}$ a realization from a uniform $[0,1]^5$ distribution.  Then, for all $i \in [21]$, we let $\phi(x_{(t)},w_{i,t})$ denote the value of $v_3$ from the integration of the ODE at time $[w_{i,j}]_0$ when using the initial condition $[w_{i,t}]_{1:3}$ and set
\begin{equation*}
y_{i,t} \gets \phi(x_{(t)},w_{i,t}) + |\phi(x_{(t)},w_{i,t})| u_{i,t},
\end{equation*}
where $u_{i,t}$ is a realization from a uniform $[-0.1,0.1]$ distribution.  Overall, we have established the data $\{(w_{i,t},y_{i,t})\}_{i\in[21]}$ for problem $t$, which is defined as in \eqref{eq:opt_least_square} with $\phi$ defined as above.  By generating the problem data in this manner for $t \in \{0,\dots,T-1\}$, each optimization problem is similar, but different due to the randomization of the initial conditions and the noise in the measurement data.

To understand the typical behavior of our algorithm, we generated a total of $N=100$ macro replications, each one with a sequence of $T=100$ instances as described in the previous paragraph.  Each repetition starts with problem $t=0$, which involves no history of function evaluations, meaning that $\alg_{\HHH}$ and $\alg_{\emptyset}$ always perform equivalently for $t=0$.
However, for all $t \in \{1,\dots,T-1\}$, $\alg_{\HHH}$ makes use of prior function evaluations when possible while $\alg_{\emptyset}$ does not.  Let $f_{\HHH,t}^k$ and $f_{\emptyset,t}^k$, respectively, denote the final objective function value obtained by $\alg_{\HHH}$ and $\alg_{\emptyset}$ when solving repetition $k$ of problem $t$.  Averaging over the macro replications, we obtain the values $\bar{f}_{\HHH,t} = \tfrac1N \sum_{k\in[N]} f_{\HHH,t}^k$ and $\bar{f}_{\emptyset,t} = \tfrac1N \sum_{k\in[N]} f_{\emptyset,t}^k$ for all $t \in \{0,\dots,T-1\}$.  In addition, to get a sense of $\alg_{\HHH}$'s ability to take more steps within the function evaluation limit by using approximated function values in place of true function values, we record $M_t^k$ as the number of approximated function values used in instance $t$ of replication $k$.  These are averaged over the replications to obtain $\overline M_t = \tfrac1N \sum_{k\in [N]} M_t^k$ for all $t \in \{0,\dots,T-1\}$.

Figure~\ref{fig.comparison2} presents the results of our experiments.  For all $t \in \{0,\dots,T-1\}$, the plot on the left shows $\sum_{\bar{t}\in[t]} (\bar f_{\emptyset, \bar{t}} - \bar f_{\HHH, \bar{t}})$, the accumulated improvement of $\alg_{\HHH}$ over $\alg_{\emptyset}$, as well as the surrounding interval of width $\pm \frac{1.96}{\sqrt{N}} \sigma_t$, where $\sigma_t$ is the standard deviation of $\{\sum_{\bar{t}\in[t]} (f_{\emptyset,\bar{t}} -  f_{\HHH,\bar{t}})\}_{k\in[N]}$.  The increasing trend shows that as the function evaluation history increases in size, $\alg_{\HHH}$ is continually able to obtain improved final objective values over $\alg_{\emptyset}$.  The plot on the right shows, for all $t \in \{0,\dots,T-1\}$, the average number of function values that $\alg_{\HHH}$ is able to approximate (instead of evaluate) in a run of the algorithm.  Recalling that the budget in each run is $2p(n_x+1) = 2 \times 21 \times (5+1) = 252$, the plot shows that by problem $t \approx 10$ over half of the function values used by $\alg_{\HHH}$ come from approximations rather than (expensive) actual evaluations, which allows the algorithm to take more iterations to improve the objective compared to~$\alg_{\emptyset}$.

\begin{figure}[ht]
	\caption{$\alg_{\HHH}$ and $\alg_{\emptyset}$ comparison.}
	\label{fig.comparison2}
	\centering
	\includegraphics[width=0.9\textwidth,clip=true,trim=30 32 25 30]{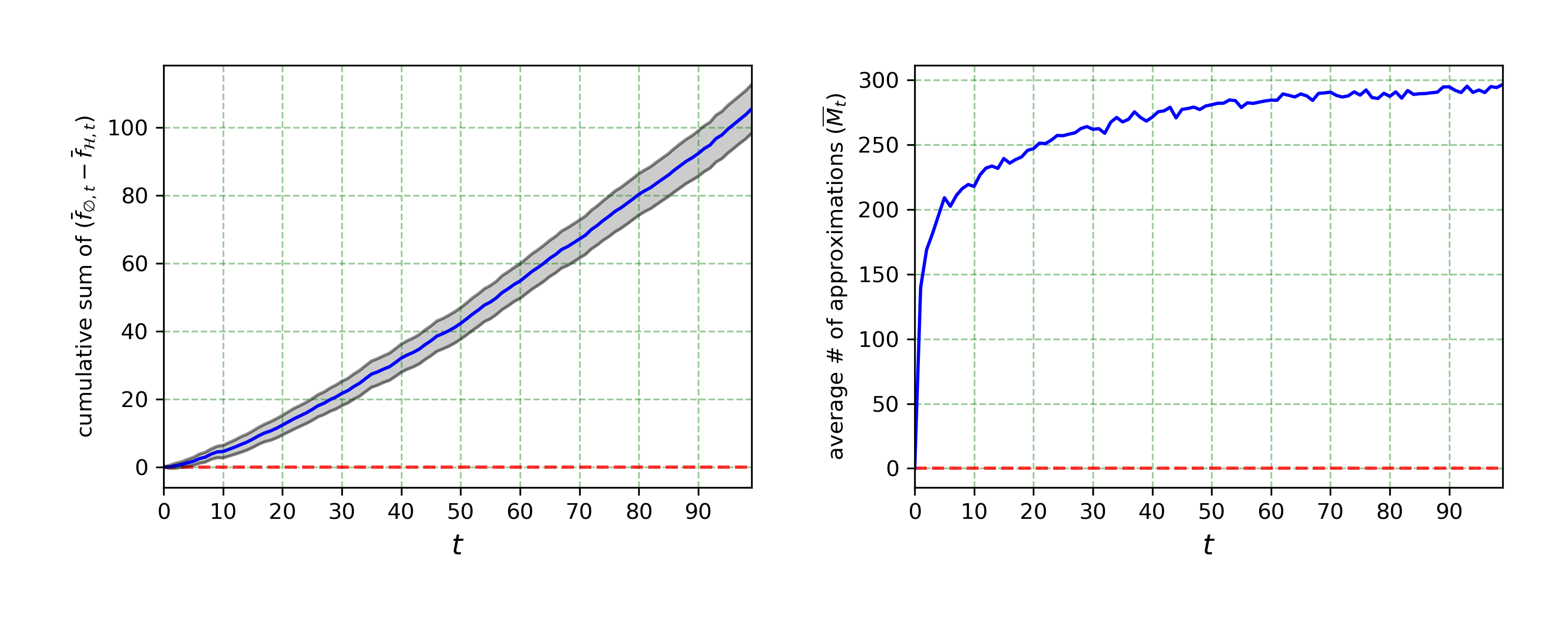}
\end{figure}

% %*********
% % Section
% %*********
{\section{Conclusion}\label{sec.conclusion}
	We proposed and analyzed a model-based DFO algorithm for solving 
	optimization problems under the assumption that the black-box objective function is smooth and black-box function evaluations are the computational bottleneck of the algorithm. The distinguishing feature of our algorithm is the use of approximate function values which can be obtained by an application-specific surrogate model that is cheap to evaluate or by using prior function values that are available from the previous solution of related optimization problems. We provided a regression-based method that approximates the objective function at interpolation points for the latter case.   In addition, we proposed an algorithm for choosing a set of well-poised interpolation points  that satisfy unrelaxable bound constraints. Our numerical results showed that our algorithm outperforms a state-of-the-art DFO algorithm for solving an engineering problem when a history of black-box function evaluations is available.}

%**************

%**************
\section{Acknowledgements}
Frank E. Curtis was supported in part by National Science Foundation Grant [NSF CCF-2008484], and Shima Dezfulian and Andreas W\"achter were suppored in part by Department of Energy Grant [ARPA-E DE-AR0001073] and National Science Foundation Grant [NSF CCF-2008484]. 
%**************
% Bibliography
%**************

%**************
% Bibliography
%**************

\bibliographystyle{tfs}
\bibliography{references.bib}
% \bibliographystyle{siamplain}
% \bibliography{references}

%**********
% Appendix
%**********
\appendix
\newpage
\section{Proofs of Lemmas~\ref{lemma:comp_fl}, \ref{lemma:bounded_hessian_m}, and \ref{lemma:master_fl}, and Theorem~\ref{th:shima}}\label{app:lemmas_proof}

\begin{proof}[Lemma~\ref{lemma:comp_fl}]
  Consider an arbitrary index $k \in \NN$, set ${\cal D}_k := \{d_1, \dots, d_{n_x}\} \subset B_{\tr}(0, \Delta_k) \cap (\Omega - x_k)$ satisfying the conditions of the lemma, point $x = x_k + s$ for some $\|s\|_\tr \leq \Delta_k$ such that $x \in \Omega$, and $i \in [p]$.  
%   For brevity, let us drop the dependence of $F_i$ on $\ttt$; i.e., let us use $F_i(x)$ to represent $F_i(x, \ttt)$. 
Following the proof of Theorem 4.1 in  \cite{stefan2008derivative}, let
  \begin{align*}
    e_{q}(s) \coloneqq q_{k,i}(x_k + s) - F_i(x_k + s)\ \ \text{and}\ \ 
    e_{\nabla q}(s) \coloneqq \nabla q_{k,i}(x_k + s) - \nabla F_i(x_k + s). 
  \end{align*}
  In addition, let $d_0 := 0$, and, for all $j \in \{0, 1, \dots, n_x \}$, let
  \begin{align*}
    I_{1,j} &\coloneqq \int_0^1 \langle \nabla F_i\big(x_k + s + t (d_j - s) \big) - \nabla F_i(x_k + s), d_j - s \rangle dt, \\
    I_{2,j} &\coloneqq \int_0^1 \langle \nabla q_{k, i}\big(x_k + s + t (d_j - s) \big) - \nabla q_{k,i}(x_k + s), d_j - s \rangle dt, \\
    I_{3,j} &\coloneqq \int_0^1 \langle \nabla F_i\big(x_k + s - t s \big) - \nabla F_i(x_k + s), - s \rangle dt, \ \ \text{and} \\
    I_{4,j} &\coloneqq \int_0^1 \langle \nabla q_{k, i}\big(x_k + s - t s \big) - \nabla q_{k, i}(x_k + s), - s \rangle dt.
  \end{align*}

  By Lemma~4.1.2 in \cite{dennis1996numerical}, it follows for all $j \in \{0, 1, \dots, n_x \}$ that
  \begin{equation}\label{eq:e_m}
    \langle e_{\nabla q}(s), d_j-s \rangle = I_{1,j} - I_{2,j} - e_{q}(s) - F_i(x_k + d_j) + q_{k, i}(x_k + d_j), 
  \end{equation}
while one also finds for all $j \in \{0, 1, \dots, n_x \}$ that
\begin{align}
  \langle e_{\nabla q}(s), d_j \rangle =&\ I_{1,j} - I_{2,j} - I_{3,j} + I_{4,j} \nonumber \\
    &\ - F_i(x_k + d_j) + q_{k, i}(x_k + d_j) + F_i(x_k) - q_{k, i}(x_k). \label{eq:l5}
\end{align}
Let us bound the integrals on the right-hand sides of \eqref{eq:e_m} and \eqref{eq:l5}.  First, under Assumption \ref{assumption:F_h}, it follows for all $j \in \{0, 1, \dots, n_x \}$ that
\begin{align}
     | I_{1, j} | 
      &\leq  \int_0^1 \|\nabla F_i\big(x_k + s + t (d_j - s) \big) - \nabla F_i(x_k + s)\|_\trc \| d_j - s\|_\tr  dt \nonumber \\
    %   \leq & \int_0^1 \nctrc \|\nabla F_i\big(x_k + s + t (d_j - s) \big) - \nabla F_i(x_k + s)\| \|d_j - s\|_\tr dt  \\
    &\leq L_{\nabla F_{i, x}}  \|d_j - s\|_\tr^2 \int_0^1  t dt 
    \leq  2  L_{\nabla F_{i, x}} \Delta_k^2. \label{eq:I1_bnd}
\end{align}
%Also, since $\|d_j - s \| \leq 2 \Delta_k$ for all $j=2, \dots, n_p+1$, the last inequality holds.
Similarly, under Assumptions~\ref{assumption:F_h} and \ref{assumption:q} it follows that
\begin{align}
\label{eq:I2,3,4_bnd}
  |I_{2, j} | 
    \leq  2  L_{\nabla q_i} \Delta_k^2, \quad     |I_{3, j}| 
%   \\
%   \leq& \int_0^1 \| \nabla F(x_i, p_k+s - ts) - \nabla  F(x_i, p_k+s) \| \| s\| dt \\
%   \leq& \int_0^1 L \| s\|^2 t dt \\
   \leq  \tfrac{1}{2} L_{\nabla F_{i, x}} \Delta_k^2, \quad |I_{4, j}| 
  \leq \tfrac{1}{2}  L_{\nabla q_i} \Delta_k^2.
\end{align}
% and for \eqref{eq:l3} we have
% \begin{align}
% \label{eq:F_bnd}
% \begin{split}
%   \Big |\int_0^1 \langle \nabla F\big(x_i, p_k + s - t s \big) - \nabla F(x_i, p_k + s), - s \rangle dt \Big | 
% %   \\
% %   \leq& \int_0^1 \| \nabla F(x_i, p_k+s - ts) - \nabla  F(x_i, p_k+s) \| \| s\| dt \\
% %   \leq& \int_0^1 L \| s\|^2 t dt \\
%   \leq  \frac{1}{2} L_g \Delta_k^2,
%   \end{split}
% \end{align}
% where to get the second inequality, we used the Lipschitz continuity of the gradient of $F(x_i, \cdot)$.
% similarly, for \eqref{eq:l4} we have
% \begin{align}
% \label{eq:m_bnd}
%   \Big |&\int_0^1 \langle \nabla m_{k, i}\big(p_k + s - t s \big) - \nabla m_{i,k}(p_k + s), - s \rangle dt \Big | 
%   \leq \frac{1}{2} \kappa_{apx}\Delta_k^2,
% \end{align}
Moreover, by \eqref{eq:gen_approx_intrpl_cond} and \eqref{cond:fl-2}, it follows for all $j \in \{0, 1, \dots, n_x \}$ that
\begin{align}
\label{eq:diff_F_bnd_1}
    |F_i(x_k + d_j) - q_{k, i}(x_k + d_j)| \leq \bar \kappa_{\apx} \Delta_k^2.
\end{align}
To show \eqref{eq:fl_g}, let $\bar D \coloneqq \begin{bmatrix} d_1 & \dots & d_{n_x} \end{bmatrix}$, note \eqref{eq:l5}--\eqref{eq:diff_F_bnd_1} shows $\|\bar D^T e_{\nabla q}(s) \|_{\infty} \leq \kappa_1 \Delta_k^2$, where  $\kappa_1 = (\tfrac{5}{2} (L_{\nabla F_{i, x}} + L_{\nabla q_i}) + 2 \bar \kappa_{\apx})$, then, with \eqref{eq.equiv1} and \eqref{cond:fl-1}, see that
\begin{align}
  \|e_{\nabla q}(s) \|_{\trc} \leq \kappa_{\tr_0^*}\|\bar D^{-T}\|_2 \| \bar D ^{T} e_{\nabla q}(s) \|_{\infty} \leq \kappa_{\tr_0^*} \kappa_1  \Lambda \Delta_k =: \hat \kappa_{\eg} \Delta_k; \label{eq.firstone}
\end{align}
\eqref{eq:fl_g} follows since $s$ was chosen arbitrarily in $B_{\tr}(0; \Delta_k)$ such that $x = x_k + s \in \Omega$.  To show \eqref{eq:fl_m}, it follows from \eqref{eq.firstone}, \eqref{eq:e_m} for $d_1=0$, the fact that $I_{1,j}$ (respectively, $I_{2,j}$) reduces to $I_{3,j}$ (respectively, $I_{4,j}$) for $j=0$ (since $d_0 = 0$), and \eqref{eq:I1_bnd}--\eqref{eq:diff_F_bnd_1} that
\begin{align*}
    |e_{q}(s)| &\leq \| e_{\nabla q}(s) \|_{\trc} \| s\|_{\tr} + \tfrac{1}{2} (L_{\nabla F_{i,x}} + L_{\nabla q_i}) \Delta_k^2 + \bar \kappa_{\apx} \Delta_k^2 \\
    % &\leq \left( \tfrac12 (5\kappa_{\tr_0^*} \Lambda + 1) (L_{\nabla F_{i, x}} + L_{\nabla q_i})+ (2\kappa_{\tr_0^*} \Lambda + 1) \bar \kappa_\apx\right) \Delta_k^2 = \hat \kappa_{\ef} \Delta_k^2;
    & \leq \left(\hat \kappa_{\eg} + \frac{1}{2} (L_{\nabla F_{i,x}} + L_{\nabla q_i}) + \bar \kappa_{\apx} \right) \Delta_k^2 =: \hat \kappa_{\ef} \Delta_k^2;
\end{align*}
\eqref{eq:fl_m} follows since $s$ was chosen arbitrarily in $B_{\tr}(0; \Delta_k)$ with $x = x_k + s \in \Omega$.
\end{proof}

\begin{proof}[Proof of Lemma~\ref{lemma:bounded_hessian_m}]
  Under Assumption~\ref{assumption:q}, it follows that $\|\nabla^2 q_{k, i}(x_k)\|_{(\tr, \trc)} \leq L_{\nabla q_i}$ for all $k \in \NN$ and $i \in [p]$ and $\|\nabla^2 h(q_k(x_k))\|_2 \leq L_{\nabla h}$ for all $k \in \NN$.  Consequently, by \eqref{eq.equiv3}, \eqref{eq:q_consts}, and the definition of $m_k$ in \eqref{eq:master_model}, one finds that
 \begin{align*}
     &\ \|\nabla^2 m_k(x_k) \|_{(\tr, \trc)} \\
     &\leq\ \sum_{i=1}^p |\partial_i h(q_k(x_k))| \| \nabla^2 q_{k, i} (x_k) \|_{(\tr, \trc)} + \kappa_{\tr_2^*}\| \nabla q_k(x_k) \|_{\trc, 1}^2 \|\nabla^2 h(q_k(x_k)) \|_2 \\
     &\leq\ \sum_{i=1}^p (\kappa_{\partial_i h} L_{\nabla q_i})  + \kappa_{\tr_2^*}  \kappa_{\nabla q}^2 L_{\nabla h} =: \kappa_{\bhm},
 \end{align*}
  as desired.
\end{proof}

\begin{proof} [Proof of Lemma~\ref{lemma:master_fl}]
  Consider an arbitrary index $k \in \NN$, set $\DDD_k := \{d_1,\dots,d_{n_x}\} \subset B_{\tr}(0, \Delta_k) \cap (\Omega - x_k)$ satisfying the conditions of the lemma, point $x = x_k + s$ for some $\|s\|_\tr \leq \Delta_k$ such that $x \in \Omega$, and $i \in [p]$. 
%   For brevity, let us drop the dependence of $F$ and $F_i$ on $\ttt$; i.e., let us use $F(x)$ and $F_i(x)$ to represent  $F(x, \ttt)$ and $F_i(x, \ttt)$, respectively.  
To show the first bound in \eqref{eq:main_thm_fcn_bnd}, observe by the Mean Value Theorem that \begin{align}
       f(x) =&\ f(x_k) + \nabla f(x_k + \tau_1 s)^Ts \label{eq:taylor_f} \\
       =&\ h(F(x_k)) + \nabla h(F(x_k + \tau_1 s))^T \nabla F(x_k + \tau_1 s)^T s\ \ \text{for some $\tau_1 \in [0, 1]$}. \nonumber
  \end{align}
  Similarly, for the master model $m_k$ in \eqref{eq:master_model}, it follows for some $\tau_2 \in [0, 1]$ that
  \begin{align}\label{eq:taylor_m}
    m_k(x) = h(q_k(x_k)) + \nabla h(q_k(x_k + \tau_2 s))^T \nabla{q_k}(x_k + \tau_2 s)^T s.
  \end{align}
  Therefore, by \eqref{eq:taylor_f} and \eqref{eq:taylor_m}, one finds that
  \begin{multline}\label{eq:abs_diff_fcn_master_mod}
      |f(x) - m_k(x)| \leq  \big| h(F(x_k)) - h(q_k(x_k))\big| \\
    + \|\nabla F(x_k + \tau_1 s) \nabla h(F(x_k + \tau_1 s)) - \nabla {q_k}(x_k + \tau_2 s) \nabla h(q_k(x_k + \tau_2 s)) \|_{\trc} \|s\|_\tr.
  \end{multline}
  To bound the first term on the right-hand side of \eqref{eq:abs_diff_fcn_master_mod}, one finds from \eqref{eq:gen_approx_intrpl_cond} and \eqref{cond:fl-2} and Assumptions \ref{assumption:F_h} and \ref{assumption:q} that 
%   \begin{align}\label{eq:fcn_diff_bnd}
%     |h(F(x_k)) - h(q_k(x_k))| \leq L_h \| F(x_k) - q_k(x_k) \|_2 \leq \sqrt{p} L_h \bar \kappa_{\app} \Delta_k^2
%   \end{align}
  \begin{align}\label{eq:fcn_diff_bnd}
    |h(F(x_k)) - h(q_k(x_k))| \leq L_h \| F(x_k) - q_k(x_k) \|_2 \leq \kappa_2 \Delta_k^2
  \end{align}
  where $\kappa_2 = \sqrt{p} L_h \bar \kappa_{\app}$.
  To bound the second term, recall that $\| s\|_\tr \leq \Delta_k$ and $\Delta_k \leq \Delta_{\max}$ for all $k \in \NN$; in addition, by \eqref{eq.equiv2}, \eqref{eq:F_x_lipschitz_consts}, \eqref{eq:q_consts}, the fact that $|\tau_1 - \tau_2| \leq 1$, Lemma \ref{lemma:comp_fl}, and Assumptions \ref{assumption:F_h} and  \ref{assumption:q} it follows that
  \begin{align}
    &\ \|\nabla F(x_k + \tau_1 s)  \nabla h(F(x_k + \tau_1 s)) -  \nabla q_k(x_k + \tau_2 s) \nabla h(q_k(x_k + \tau_2 s))  \|_{\trc} \label{eq:grad_diff_bnd} \\
    \leq&\ \kappa_{\tr_1^*} \| \nabla F(x_k + \tau_1 s) - \nabla q_k(x_k + \tau_1 s) \|_{\trc, 1} \| \nabla h(F(x_k + \tau_1 s)) \|_2 \nonumber \\
    &\ +  \kappa_{\tr_1^*} \| \nabla q_k(x_k + \tau_1 s) - \nabla q_k(x_k + \tau_2 s) \|_{\trc, 1} \| \nabla h(q_k(x_k + \tau_2 s))\|_2 \nonumber \\
    &\ + \kappa_{\tr_1^*} \| \nabla q_k(x_k + \tau_1 s) \|_{\trc, 1} \| \nabla h(F(x_k + \tau_1 s)) - \nabla h(q_k(x_k + \tau_2 s)) \|_2 \nonumber \\
    \leq&\ p \kappa_{\tr_1^*} \hat \kappa_{\eg}  \|\nabla h(F(x_k + \tau_1 s)) \|_2 \Delta_k  + \kappa_{\tr_1^*} L_{\nabla q} \| \nabla h(q_k(x_k + \tau_2 s)) \|_2 \Delta_k \nonumber \\
    &\ + \kappa_{\tr_1^*} L_{\nabla h} \| \nabla q_k(x_k + \tau_1 s) \|_{\trc, 1} \| F(x_k + \tau_1 s) - F(x_k + \tau_2 s) \|_2 \nonumber \\
    &\ + \kappa_{\tr_1^*} L_{\nabla h} \| \nabla q_k(x_k + \tau_1 s) \|_{\trc, 1} \|F(x_k + \tau_2 s) -  q_k(x_k + \tau_2 s) \|_2  
    % \leq&\ p \kappa_{\tr_1^*}  \hat \kappa_{\eg}  \|\nabla h(F(x_k + \tau_1 s)) \|_2 \Delta_k  + \kappa_{\tr_1^*} L_{\nabla q} \| \nabla h(q(x_k + \tau_2 s)) \|_2  \Delta_k \\
    % &\ + \kappa_{\tr_1^*} L_{\nabla h}  \big( \bar L_{F_x} + \sqrt{p}  \hat \kappa_{\ef}  \Delta_k \big)  \| \nabla q_k(x_k + \tau_1 s) \|_{\trc, 1}  \Delta_k.
    \leq  \kappa_3 \Delta_k;  \nonumber 
  \end{align}
  where $\kappa_3 := \ \kappa_{\tr_1^*}  \left ( \kappa_{\nabla h} (p  \hat \kappa_{\eg}  +  L_{\nabla q})  + L_{\nabla h}  \kappa_{\nabla q} \left(\bar L_{F_x} + \sqrt{p} \hat \kappa_{\ef} \Delta_{\max}\right)\right)$.  From \eqref{eq:abs_diff_fcn_master_mod}--\eqref{eq:grad_diff_bnd}
  %and the fact that $\Delta_k \leq \Delta_{\max}$ for all $k \in \NN$, 
  it follows that the first bound in \eqref{eq:main_thm_fcn_bnd} holds for $\kappa_{\ef} = \kappa_2 + \kappa_3$, as desired.  Next, to show the second bound in \eqref{eq:main_thm_fcn_bnd}, first observe by Taylor's Theorem that
\begin{align}
\label{eq:taylor_grad_f}
    \nabla f(x) = 
    % & \nabla f(x_k) + \int_0^1 \nabla^2 f(x_k + t s)^T s d t\ = \
    \nabla F(x_k) \nabla h(F(x_k)) + \int_0^1 \nabla^2 f(x_k + \tau s) s d\tau,
\end{align}
where 
\begin{multline}\label{eq:hessian_f}
  \nabla^2 f(x_k + \tau s) = \sum_{i=1}^p \partial_i h(F(x_k + \tau s)) \nabla^2 F_i(x_k + \tau s) \\
  + \nabla F(x_k + \tau s) \nabla^2 h(F(x_k + \tau s)) \nabla F(x_k + \tau s)^T.
\end{multline}
Also, from \eqref{eq:master_model}, it follows that
\begin{align}
\label{eq:taylor_grad_m}
    \nabla m_k(x) =&  \nabla {q_k}(x_k) \nabla h(q_k(x_k)) 
    + \nabla^2 m_k(x_k) s. %\Big(\sum_{i=1}^p \partial_i h(q_k(x_k)) \nabla^2 q_{k, i}(x_k) + \nabla {q_k}(x_k) \nabla^2 h(q_k(x_k)) \nabla {q_k}(x_k)^T\Big) s
\end{align}
As a result, from \eqref{eq:taylor_grad_f} and \eqref{eq:taylor_grad_m}, one obtains
\begin{multline}\label{eq:grad_m_f_diff}
     \| \nabla m_k(x) - \nabla f(x) \|_{\trc} \leq   \left \| \nabla F(x_k) \nabla h(F(x_k)) - \nabla q_k(x_k) \nabla h(q_k(x_k)) \right \|_{\trc} \\
    + \left \| \int_{0}^1  \nabla^2 f(x_k + \tau s) s d \tau \right \|_{\trc} +  \| s\|_\tr  \| \nabla^2 m_k(x_k)\|_{(\tr, \trc)}.
\end{multline}
Let us now bound each of the terms on the right-hand side of \eqref{eq:grad_m_f_diff}. For the first term, by Assumptions \ref{assumption:F_h} and \ref{assumption:q} and Lemma \ref{lemma:comp_fl}, and \eqref{eq:F_x_lipschitz_consts}, one has that
\begin{align}\label{eq:first_term_grad_diff}
  &\ \| \nabla F(x_k) \nabla h(F(x_k)) - \nabla q_k(x_k) \nabla h(q_k(x_k)) \|_{\trc} \\ 
  \leq&\ \| \nabla F(x_k) \nabla h(F(x_k)) - \nabla F(x_k) \nabla h(q_k(x_k)) \|_{\trc} \nonumber \\
  &\ + \| \nabla F(x_k) \nabla h(q_k(x_k)) - \nabla q_k(x_k) \nabla h(q_k(x_k)) \|_{\trc}\nonumber \\
  \leq&\ \kappa_{\tr_1^*} \| \nabla F(x_k) \|_{\trc, 1} \| \nabla h(F(x_k)) - \nabla h(q_k(x_k)) \|_2 \nonumber \\
  &\ + \kappa_{\tr_1^*} \| \nabla F(x_k) - \nabla q_k(x_k) \|_{\trc, 1} \| \nabla h(q_k(x_k)) \|_2 \nonumber \\
  \leq&\ \kappa_{\tr_1^*} \sqrt{p} L_{F_x} L_{\nabla h} \bar \kappa_{\app} \Delta_k^2 + \kappa_{\tr_1^*} p \kappa_{\nabla h}  \hat \kappa_{\eg}  \Delta_k 
%   \leq&\ \kappa_{\tr_1^*} \sqrt{p} \big(L_{F_x} L_{\nabla h} \bar \kappa_{\app} \Delta_{\max} + \sqrt{p} \kappa_{\nabla h} \hat \kappa_{\eg} \big) \Delta_k.
\leq \kappa_4 \Delta_k \nonumber 
\end{align}
where $\kappa_4 := \kappa_{\tr_1^*} \sqrt{p} \big(L_{F_x} L_{\nabla h} \bar \kappa_{\app} \Delta_{\max} + \sqrt{p} \kappa_{\nabla h} \hat \kappa_{\eg} \big)$.  For the second term, one finds
\begin{align*}
 \left \| \int_0^1 \nabla^2 f(x_k + \tau s)  s d \tau \right \|_{\trc} \leq %& \int_0^1   \left \| \nabla^2 f(x_k + \tau s)  s \right \|_{\trc} d \tau \\ 
%  \leq &  \int_0^1   \left \| \nabla^2 f(x_k + \tau s) \right \|_{(\tr, \trc)} \left \| s \right \|_{\tr} d \tau  \\
%  =& 
 %\leq & 
 \left \| s \right \|_{\tr} \int_0^1 \left \| \nabla^2 f(x_k + \tau s) \right \|_{(\tr, \trc)} d \tau,
\end{align*}
where, by \eqref{eq:hessian_f}, \eqref{eq.equiv3}, \eqref{eq:F_x_lipschitz_consts}, and Assumptions \ref{assumption:F_h} and  \ref{assumption:q}, it follows that
\begin{multline}\label{eq:second_term_hessian_f}
  \|  \nabla^2 f(x_k + \tau s)  \|_{(\tr, \trc)}
    \leq \sum_{i=1}^p  |\partial_i h(F(x_k + \tau s))|  \left \| \nabla^2 F_i(x_k + \tau s)  \right \|_{(\tr, \trc)}\\
    + \kappa_{\tr_2^*} \|\nabla F(x_k + \tau s) \|_{\trc, 1}^2 \|\nabla^2 h(F(x_k + \tau s))\|_2 \leq  \kappa_5.
\end{multline}
where $\kappa_5 := \sum_{i=1}^p (\kappa_{\partial_i h}L_{\nabla F_{i, x}}) + \kappa_{\tr_2^*} L_{F_x}^2 L_{\nabla h}$.  For the third term, by Lemma \ref{lemma:bounded_hessian_m},
\begin{align}
 \label{eq:third_term_hessian_m}
 \| \nabla^2 m_k(x_k) \|_{(\tr, \trc)} \leq \kappa_{\bhm}.
\end{align}
Substituting \eqref{eq:first_term_grad_diff}--\eqref{eq:third_term_hessian_m} into \eqref{eq:grad_m_f_diff}, one obtains
\begin{align*}
   \| \nabla m_k(x) - \nabla f(x) \|_{\trc} \leq (\kappa_4 + \kappa_5 + \kappa_{\bhm}) \Delta_k =: \kappa_{\eg} \Delta_k
%   \leq&\ \Big(\kappa_{\tr_1^*} \sqrt{p} \big(L_{F_x} L_{\nabla h} \bar \kappa_{\app} \Delta_{\max} + \sqrt{p} \kappa_{\nabla h} \hat \kappa_{\eg} \big)\\
%   &+ \sum_{i=1}^p \big(\kappa_{\partial_i h}L_{\nabla F_{i, x}} \big) 
%   +\kappa_{\tr_2^*} L_{\nabla h} L_{F_x}^2 + \kappa_{\bhm} \Big)\Delta_k
\end{align*}
as desired.
\end{proof}

\begin{proof}[Proof of Theorem~\ref{th:shima}]
  Since $|\DDD_k| = n_x+1$, let us express $[\DDD_k]_{1:{\rm end}} = \begin{bmatrix} d_1 & \cdots & d_{n_x} \end{bmatrix}$, where it follows from $\EEE \subset B_{\tr}(x_k, \Delta_k)$, the manner in which $\DDD_k$ is constructed, and \eqref{eq.equiv7} that $\|d_i\|_2 \leq \kappa_{\tr_0} \|d_i\|_{\tr} \leq \kappa_{\tr_0} \Delta_k$  for all $i \in [n_x]$.  In addition, let $\{\sigma_i\}_{i=1}^{n_x}$ be the singular values of $\frac{1}{\kappa_{\tr_0} \Delta_k} [\DDD_k]_{1:{\rm end}}$ such that $\sigma_1 \leq \cdots \leq \sigma_{n_x}$.  One finds that
  \begin{align}\label{eq:D_inv_bnd1}
    \|[\DDD_k]_{1:{\rm end}}^{-1}\|_2 = \frac{1}{ \kappa_{\tr_0} \Delta_k} \left\| \left(\frac{[\DDD_k]_{1:{\rm end}}}{\kappa_{\tr_0} \Delta_k} \right)^{-1}\right\|_2 = \frac{1}{\sigma_1 \kappa_{\tr_0} \Delta_k}.
  \end{align}
  In addition, letting $QR$ denote a QR factorization of $\frac{1}{\Delta_k} [\DDD_k]_{1:{\rm end}}$, it follows that the determinant of $\frac{1}{\Delta_k} [\DDD_k]_{1:{\rm end}}$ is equal to the product of the diagonal elements of~$R$, call them $\{r_i\}_{i=1}^{n_x}$.  
  Recalling $\|d_i\|_2 \leq \kappa_{\tr_0} \Delta_k$, we have
  \[
  \sigma_{n_x} = \left \| \frac{[\DDD_k]_{1: {\rm end}}}{\kappa_{\tr_0} \Delta_k}  \right \|_2 \leq \left \| \frac{[\DDD_k]_{1: {\rm end}}}{\kappa_{\tr_0} \Delta_k}  \right \|_F \leq \sqrt{n_x}.
  \]
  Since Lemma~\ref{lemma:proj_norm_QR} ensures that $|r_i| \geq \xi$ for all $i \in [n_x]$, it follows that
  \begin{align*}
    \sigma_1 n_x^{\frac{n_x - 1}{2}} \geq\sigma_1 \sigma_{n_x}^{n_x - 1} \geq \prod_{i=1}^{n_x} \sigma_i =  \left| \det \left(\frac{[\DDD_k]_{1:{\rm end}}}{\kappa_{\tr_0} \Delta_k}\right) \right| =  \left| \prod_{i=1}^{n_x} \frac{r_{i}}{\kappa_{\tr_0}}  \right| \geq \left(\frac{\xi}{\kappa_{\tr_0}}\right)^{n_x}.
  \end{align*}
  These inequalities along with \eqref{eq:D_inv_bnd1} yields
  \begin{align*}
    \| [\DDD_k]_{1:{\rm end}}^{-1} \|_2 \leq \frac{n_x^{\frac{n_x - 1}{2}} \kappa_{\tr_0}^{n_x - 1}}{ \xi^{n_x}\Delta_k} = \frac{\Lambda }{\Delta_k},
  \end{align*}
  which is the desired conclusion.
\end{proof}

\section{Proof of Theorem~\ref{th.main}} \label{app:conv_proof}

Our first lemma shows that for any model $m_k$ of $f$ that is constructed, one has that the difference in the stationarity measures with respect to $m_k$ and $f$ are proportional to the radius of the trust region in which the model is fully linear.

\begin{lemma}\label{lem.pi}
 For all $k \in \NN$, it follows that for any value of $\Delta_k$ such that a model~$m_k$ is constructed, one has that $|\pi_k^f - \pi_k^m| \leq \kappa_{\eg} \Delta_k$.
\end{lemma}
\begin{proof}
  Consider arbitrary $k \in \NN$.  It follows by construction of Algorithm~\ref{alg.dfo_approx} and Lemma~\ref{lemma:master_fl} that, for any value of $\Delta_k$ for which a model $m_k$ is constructed, $m_k$ is fully linear with respect to $f$ in $B(x_k,\Delta_k) \cap \Omega$, which along with \cite[Lemma~7]{conn1993global} shows that $|\pi_k^f - \pi_k^m | \leq \|\nabla f(x_k) - \nabla m_k(x_k)\|_{\trc} \leq \kappa_{\eg} \Delta_k$, as desired.
\end{proof}

We now prove that if a criticality step is not performed in iteration $k \in \NN$, then the stationarity measure with respect to $m_k$ is bounded below in terms of the trust region radius, and if in addition the stationarity measure with respect to $f$ is bounded below by a positive constant, then the stationarity measure with respect to $m_k$ is similarly bounded below by a positive constant.

\begin{lemma}\label{lem.criticality_not_called}
  For any $k \in \NN$, if the condition in line~\ref{st:check_criticality} does not hold, then $\pi_k^m \geq \min\{\epsilon_c, \mu^{-1}\Delta_k\}$.  If, in addition, $\pi_k^f \geq \epsilon \in \RR_{>0}$, then 
  \begin{align*}
    \pi_k^m \geq \epsilon_{\mc} := \min \left\{\epsilon_c, \frac{\epsilon}{1 + \kappa_{\eg} \mu} \right\} \in \RR_{>0}.
  \end{align*}
\end{lemma}
\begin{proof}
  Consider arbitrary $k \in \NN$.  If the condition in line~\ref{st:check_criticality} does not hold, then one has that $\pi_k^m > \epsilon_c$ or $\mu \pi_k^m \geq \Delta_k$; hence, $\pi_k^m \geq \min\{\epsilon_c, \mu^{-1}\Delta_k\}$, as desired.  Now suppose in addition that $\pi_k^f \geq \epsilon$.  If $\pi_k^m \geq \epsilon_c$, then there is nothing left to prove.  Otherwise, since the condition in line~\ref{st:check_criticality} does not hold, it follows that $\mu \pi_k^m \geq \Delta_k$, which along with Lemma~\ref {lem.pi} and the full linearity of $m_k$ yields
  \begin{align*}
    \epsilon \leq \pi_k^f \leq | \pi_k^f - \pi_k^m | + \pi_k^m \leq \kappa_{\eg} \Delta_k + \pi_k^m \leq (\kappa_{\eg} \mu +1) \pi_k^m.
  \end{align*}
  Combining the results from the two cases, the desired conclusion follows.
\end{proof}

Next, we show that if $\Delta_k$ is sufficiently small, then a successful step occurs.

\begin{lemma}\label{lem.delta_small_succ}
  For any $k \in \NN$, if trust region radius satisfies
  \begin{align}\label{eq:c_0}
    \Delta_k \leq \min\{c_0 \pi_k^m, 1\}, \ \ \text{where} \ \ c_0 := \min \left\{\mu, \frac{1}{\kappa_{\bhm} + 1}, \frac{\kappa_{\fcd}(1-\eta)}{2 \kappa_{\ef}} \right\},
  \end{align}
  then the condition in line~\ref{st:check_criticality} does not hold and $\rho_k \geq \eta$, i.e., the step is successful.
\end{lemma}
\begin{proof}
  Consider arbitrary $k \in \NN$ such that \eqref{eq:c_0} holds.  Since $\Delta_k \leq c_0 \pi_k^m$, where $c_0 \leq \mu$, it follows that $\Delta_k \leq \mu \pi_k^m$, from which it follows that the condition in line~\ref{st:check_criticality} does not hold, as desired.  In addition, by \eqref{eq:cauchy_dec} and \eqref{eq:c_0}, it follows that
  \begin{align*}
    m_k(x_k) - m_k(x_k + s_k) \geq \kappa_{\fcd} \pi_k^m \min \left\{\frac{\pi_k^m}{\kappa_{\bhm} + 1}, \Delta_k, 1 \right\} = \kappa_{\fcd} \pi_k^m \Delta_k.
  \end{align*}
  Therefore, by \eqref{eq:act_pred_ratio} and \eqref{eq:main_thm_fcn_bnd}, one has that 
  \begin{align*}
    |\rho_k - 1 | &\leq \left| \frac{f(x_k) - m_k(x_k)}{m_k(x_k) - m_k(x_k + s_k)} \right| + \left | \frac{f(x_k + s_k) - m_k(x_k + s_k)}{m_k(x_k) - m_k(x_k + s_k)} \right| \leq \frac{2 \kappa_{\ef} \Delta_k}{\kappa_{\fcd} \pi_k^m},
  \end{align*}
  which, since \eqref{eq:c_0} ensures $\Delta_k \leq c_0\pi_k^m \leq \frac{\kappa_{\fcd}(1-\eta)}{2\kappa_{\ef}}$, implies that $|\rho_k - 1 | \leq 1 - \eta$.  Hence, $\rho_k \geq \eta$ and iteration $k$ yields a successful step, as desired.
\end{proof}

We now prove that if the stationarity measure with respect to $f$ is bounded below by a positive constant, then the trust region radius is similarly bounded below.
\begin{lemma}\label{lem.Delta_min}
  If $\pi_k^f \geq \epsilon \in \RR_{>0}$ for all $k \in \NN$, then $\Delta_k \geq \Delta_{\min}$ for all $k \in \NN$, where
  \begin{align}\label{eq:Delta_min}
    \Delta_{\min} := \min \left\{\Delta_0, \frac{\gamma_{\dec} \epsilon}{\kappa_{\eg} + \mu^{-1}}, \gamma_{\dec} \left(\kappa_{\eg} + \frac{2 \kappa_{\ef}}{\kappa_{\fcd} (1-\eta)} \right)^{-1} \epsilon, \gamma_{\dec} \mu \epsilon_{\mc}, \frac{\gamma_{\dec} \epsilon_{\mc}}{\kappa_{\bhm} + 1}, \gamma_{\dec} \right\}.
  \end{align}
\end{lemma}
\begin{proof}
  To derive a contradiction, suppose that $\pi_k^f \geq \epsilon > 0$ for all $k \in \NN$ and that there exists $\hat{k} \in \NN$ such that $\Delta_{\hat{k}} < \Delta_{\min}$.  Since $\Delta_{\min} \leq \Delta_0$, it must be true that $\hat{k} > 0$, and by the definition of $\hat{k}$ one has $\Delta_{\hat{k}} < \Delta_{\min} \leq \Delta_{\hat{k} - 1}$.  Since $\Delta_{\hat{k}} < \Delta_{\hat{k} - 1}$, iteration $\hat{k} - 1$ yields either a criticality or unsuccessful step.  If iteration $\hat{k} - 1$ yields a criticality step, then by the condition in line~\ref{st:check_criticality} it must be true that $\mu \pi_{\hat{k}-1}^m < \Delta_{\hat{k} - 1}$, which along with Lemma~\ref{lem.pi} means that
  \begin{align*}
    \epsilon \leq \pi_{\hat{k} -1}^f \leq |\pi_{\hat{k} - 1}^f - \pi_{\hat{k} -1}^m | + \pi_{\hat{k} -1}^m \leq (\kappa_{\eg} + \mu^{-1}) \Delta_{\hat{k} - 1}.
  \end{align*}
  Hence, along with \eqref{eq:Delta_min}, it follows that
  \begin{align}\label{eq:gamma_dec_Delta_k_prime}
    \Delta_{\hat{k}} = \gamma_{\dec} \Delta_{\hat{k} - 1} \geq \frac{\gamma_{\dec} \epsilon}{\kappa_{\eg} + \mu^{-1}} \geq \Delta_{\min},
  \end{align}
  which contradicts the assumption that $\hat{k}$ is the first iteration index with $\Delta_{\hat{k}} < \Delta_{\min}$.  Thus, iteration $\hat{k} - 1$ does not yield a criticality step.  Now suppose that iteration $\hat{k} - 1$ yields an unsuccessful step, meaning that $\rho_{\hat{k} - 1} < \eta$.  Since $\pi_k^f \geq \epsilon > 0$ and the condition in line~\ref{st:check_criticality} does not hold in iteration $\hat{k}-1$ (since the iteration yields an unsuccessful step), it follows by Lemma~\ref{lem.criticality_not_called} that $\pi_{\hat{k} - 1}^m \geq \epsilon_{\mc}$.  We now claim that
  \begin{equation}\label{eq.delta_lower_proof}
    \Delta_{\hat{k} - 1} \leq  \frac{\kappa_{\fcd}(1- \eta) \pi_{\hat{k} -1}^m}{2 \kappa_{\ef}}.
  \end{equation}
  After all, if \eqref{eq.delta_lower_proof} does not hold, then it follows with Lemma~\ref{lem.pi} that
  \begin{align*}
    \epsilon \leq \pi_{\hat{k} -1}^f \leq |\pi_{\hat{k}- 1}^m - \pi_{\hat{k} - 1}^f| + \pi_{\hat{k} - 1}^m \leq \kappa_{\eg} \Delta_{\hat{k} - 1} + \pi_{\hat{k} - 1}^m < \left(\kappa_{\eg} + \frac{2 \kappa_{\ef}}{\kappa_{\fcd}(1 - \eta)} \right) \Delta_{\hat{k} - 1},
  \end{align*}
  but that contradicts the fact that iteration $\hat{k} - 1$ yielding an unsuccessful step and the definition of $\Delta_{\min}$ in \eqref{eq:Delta_min} together imply that
  \begin{align*}
    \Delta_{\hat{k} -1} = \gamma_{\dec}^{-1} \Delta_{\hat{k}} < \gamma_{\dec}^{-1} \Delta_{\min } \leq \left(\kappa_{\eg} + \frac{2 \kappa_{\ef}}{\kappa_{\fcd}(1 - \eta)} \right)^{-1 } \epsilon.
  \end{align*}
  Hence, \eqref{eq.delta_lower_proof} holds.  Combining this with the fact that \eqref{eq:Delta_min} also yields
  \begin{align*}
    \Delta_{\hat{k} -1} = \gamma_{\dec}^{-1} \Delta_{\hat{k}} < \gamma_{\dec}^{-1} \Delta_{\min} \leq \min \left\{\mu \epsilon_{\mc}, \frac{\epsilon_{\mc}}{\kappa_{\bhm} + 1}, 1 \right\},
  \end{align*}
  the fact that Lemma~\ref{lem.criticality_not_called} yields $\pi_k^m \geq \epsilon_{\mc}$, and $c_0$ from \eqref{eq:c_0}, it follows that
  \begin{align*}
    \Delta_{\hat{k} - 1} \leq \min \left\{\mu \pi_{\hat{k}-1}^m, \frac{\pi_{\hat{k}-1}^m}{\kappa_{\bhm} + 1}, \frac{\kappa_{\fcd}(1- \eta) \pi_{\hat{k} -1}^m}{2 \kappa_{\ef}}, 1 \right\} = \min \{c_0 \pi_{\hat{k}-1}^m , 1\}.
  \end{align*}
  Therefore, by Lemma~\ref{lem.delta_small_succ}, one deduces that $\rho_{\hat{k} - 1} \geq \eta$ which contradicts the assumption that iteration $\hat{k} -1$ yields an unsuccessful step.  Overall, we have reached a contradiction to the existence of $\hat{k} \in \NN$ such that $\Delta_{\hat{k}} < \Delta_{\min}$, as desired.
\end{proof}

Our next lemma shows that if the number of successful steps is finite, then
the trust region radius and stationarity measure with respect to $f$ must vanish.
\begin{lemma}\label{lem.finite_succ}
  If $|\{k \in \NN : \rho_k \geq \eta\}|$ is finite, then $\displaystyle \lim_{k \to \infty} \Delta_k = 0$ and $\displaystyle \lim_{k \to \infty} \pi_k^f = 0$.
\end{lemma}
\begin{proof}
  Suppose that $k_s$ is the largest element of $\NN$ such that iteration $k_s$ yields a successful step.  Then, for all $k \in \NN$ with $k \geq k_s$, iteration $k$ yields a criticality or unsuccessful step.  This means that $\Delta_{k+1} = \gamma_{\dec} \Delta_k$ for all $k \in \NN$ with $k \geq k_s$, from which it follows that $\{\Delta_k\} \to 0$, as desired.  Now, if $\{\pi_k^m\} \not\to 0$, then it follows from Lemma~\ref{lem.delta_small_succ} that there exists $\hat{k} \in \NN$ with $\hat{k} > k_s$ such that iteration $\hat{k}$ yields a successful step.  Since this contradicts the definition of $k_s$, it follows that $\{\pi_k^m\} \to 0$.  Combining $\{\Delta_k\} \to 0$, $\{\pi_k^m\} \to 0$, and the fact that, for all $k \in \NN$, Lemma~\ref{lem.pi} implies $\pi_k^f \leq | \pi_k^f - \pi_k^m | + \pi_k^m \leq \kappa_{\eg} \Delta_k + \pi_k^m$, it follows that $\{\pi_k^f\} \to 0$, as desired.
\end{proof}

We now show that the trust region radius always vanishes.

\begin{lemma}\label{lem.Delta_0}
  The trust region radius vanishes, i.e., $\displaystyle \lim_{k \to \infty} \Delta_k = 0$.
\end{lemma}
\begin{proof}
  If number of successful steps is finite, then the desired conclusion follows from Lemma~\ref{lem.finite_succ}.  Hence, one may proceed under the assumption that there are infinite number of successful steps.  Let $\cal S$ denote the set of iterations yielding successful steps.  By \eqref{eq:cauchy_dec}, it follows for any $k \in \cal S$ that
  \begin{align*}
    f(x_k) - f(x_{k+1}) \geq \eta (m_k(x_k) - m_k(x_k + s_k)) \geq \eta \kappa_{\fcd} \pi_k^m \min \left\{\frac{\pi_k^m}{\kappa_{\bhm} + 1}, \Delta_k, 1 \right\}.
  \end{align*}
  Therefore, since the condition in line~\ref{st:check_criticality} does not hold for any $k \in \cal S$, it follows with Lemma~\ref{lem.criticality_not_called} that, for such $k$, one has $\pi_k^m \geq \min\{\epsilon_c, \mu^{-1}\Delta_k\}$ and
  \begin{align*}
    \sum_{k \in \cal S} f(x_k) - f(x_{k + 1}) \geq \sum_{k \in S} \eta \kappa_{\fcd} \min\{\epsilon_c,\mu^{-1}\Delta_k\} \min \left\{\frac{\min\{\epsilon_c, \mu^{-1}\Delta_k\}}{\kappa_{\bhm} + 1} , \Delta_k, 1\right\} \geq 0.
  \end{align*}
  Hence, since $|\cal S|$ is infinite and $f$ is bounded below under Assumption~\ref{assumption:bounded_f}, it follows that $\{\Delta_k\}_{k\in\cal S} \to 0$.  Now, for any $k \in \NN$, observe that $\Delta_k \leq \gamma_{\inc} \Delta_{s(k)}$, where the index $s(k)$ corresponds to the largest index in $\cal S$ such that $s(k) \leq k$ (i.e., $s(k)$ corresponds to the latest iteration yielding a successful step up to iteration $k$).  Then, it follows from $\{\Delta_k\}_{k\in\cal S} \to 0$ that, in fact, $\{\Delta_k\} \to 0$, as desired.
\end{proof}

The next lemma shows that a subsequence of stationary measures vanishes.
\begin{lemma}\label{lem.liminf_f}
  The limit inferior of $\{\pi_k^f\}$ is zero, i.e., $\displaystyle \liminf_{k \to \infty} \pi_k^f = 0$.
\end{lemma}
\begin{proof}
  If the number of successful steps is finite, then the desired conclusion follows from Lemma~\ref{lem.finite_succ}.  Hence, one may proceed under the assumption that there are infinite number of successful steps.  To derive a contradiction, suppose that there exists $k_0 \in \NN$ and $\epsilon \in \RR_{>0}$ such that $\pi_k^f \geq \epsilon$ for all $k \geq k_0$.  Let $\cal S$ denote the set of iterations yielding successful steps and consider arbitrary $k \in \cal S$ with $k \geq k_0$.  By \eqref{eq:cauchy_dec}, the fact that $\pi_k^f \geq \epsilon$, and Lemmas~\ref{lem.criticality_not_called} and \ref{lem.Delta_min}, one finds that
  \begin{align*}
    f(x_k) - f(x_{k + 1}) \geq \eta (m_k(x_k) - m_k(x_k + s_k)) \geq \eta \kappa_{\fcd} \epsilon_{\mc} \min \left\{\frac{\epsilon_{\mc}}{\kappa_{\bhm} + 1}, \Delta_{\min}, 1\right\}.
  \end{align*}
  Summing over all $k \in \cal S$ with $k \geq k_0$ and using the fact that $f$ is bounded below under Assumption~\ref{assumption:bounded_f}, one reaches a contradiction.  Hence, one deduces that $\{\pi_k^f\} \to 0$.
\end{proof}

Finally, we are prepared to prove our main result, Theorem~\ref{th.main}.

\begin{proof}[Theorem~\ref{th.main}]
  Let $\cal S$ denote the set of iterations yielding successful steps.  If $|{\cal S}|$ is finite, then the desired conclusion follows by Lemma~\ref{lem.finite_succ}.  Hence, one may proceed under the assumption that $|{\cal S}|$ is infinite.  Since $x_{k+1} = x_k$, and hence $\pi_{k+1}^f = \pi_k^f$, for all $k \in \NN$ such that iteration $k$ does not yield a successful step, one may proceed by considering exclusively the indices in $\cal S$.  To arrive at a contradiction, suppose there exists infinite $\{t_i\} \subseteq \cal S$ and $\epsilon \in \RR_{>0}$ such that, for all $i \in \NN$, one has $\pi_{t_i}^f \geq 2 \epsilon > 0$.  By Lemma~\ref{lem.liminf_f}, it follows for all $i \in \NN$ that there exists some smallest $l_i \in \cal S$ with $l_i > t_i$ such that $\pi_{l_i}^f < \epsilon$.  Therefore, the subsequences $\{t_i\}$ and $\{l_i\}$ of $\cal S$ satisfy, for all $i \in \NN$,
  \begin{align}\label{eq:assum_main_thm}
    \pi_{t_i}^f \geq 2 \epsilon, \ \ \pi_k^f \geq \epsilon \ \text{for all $k \in \cal S$ such that $t_i < k < l_i$}, \ \ \text{and} \ \ \pi_{l_i}^f < \epsilon.
  \end{align}
  By Lemma~\ref{lem.criticality_not_called}, one has $\pi_k^m \geq \epsilon_{\mc} > 0$ for all $k \in {\cal{K}} := \bigcup_{i=0}^{\infty} \{k \in {\cal{S}}: t_i \leq k < l_i \}$.  Hence, by \eqref{eq:cauchy_dec}, it follows for all $k \in \cal K$ that
  \begin{align*}
    f(x_k) - f(x_{k+1}) &\geq \eta (m_k(x_k) - m_k(x_k + s_k)) \\
    &\geq \eta \kappa_{\fcd} \epsilon_{\mc} \min \left\{\frac{\epsilon_{\mc}}{\kappa_{\bhm} + 1}, \Delta_k, 1\right\} \geq 0.
  \end{align*}
  Since Lemma~\ref{lem.Delta_0} shows that $\{\Delta_k\} \to 0$, it follows for all sufficiently large $k \in \cal{K}$ that
  \begin{align*}
    \Delta_k \leq \frac{1}{\eta \kappa_{\fcd} \epsilon_{\mc}} (f(x_k) - f(x_{k+1})).
  \end{align*}
  Therefore, for all sufficiently large $i \in \NN$ one concludes that  
  \begin{align*}
    \| x_{t_i} - x_{l_i} \|_{\tr} \leq \sum_{\substack{j = t_i \\ j \in \mathcal{K}}}^{l_i - 1} \| x_j - x_{j+1}\|_{\tr} \leq \sum_{\substack{j = t_i \\ j \in \mathcal{K}}}^{l_i - 1} \Delta_j \leq \frac{1}{\eta \kappa_{\fcd} \epsilon_{\mc}} (f(x_{t_i}) - f(x_{l_i})).
  \end{align*}
  Under Assumption~\ref{assumption:bounded_f} and by construction of the algorithm, it follows that $\{f(x_k)\}$ is bounded below and monotonically nonincreasing; hence, along with the bound above, it follows that $\{\| x_{t_i} - x_{l_i} \|_{\tr}\}_{i=0}^\infty \to 0$.  Thus, by Lipschitz continuity of $\nabla f$  (with constant $L_{\nabla f}$ defined by \eqref{eq.needthis}), it follows that
  \begin{align*}
    | \pi_{t_i}^f - \pi_{l_i}^f | \leq \kappa_{\eg} \|\nabla f(x_{t_i}) - \nabla f(x_{l_i}) \|_{\trc}  \leq \kappa_{\eg} L_{\nabla f} \| x_{t_i} - x_{l_i} \|_{\tr}.
  \end{align*}
  Therefore, $\{| \pi_{t_i}^f - \pi_{l_i}^f |\}_{i=0}^\infty \to 0$. However, this is in contradiction with $| \pi_{t_i}^f - \pi_{l_i}^f | \geq \epsilon > 0$ for all $i \in \NN$ (see \eqref{eq:assum_main_thm}), and hence the desired result follows.
\end{proof}

\end{document}